\documentclass[11pt]{amsart}
\usepackage{amssymb}
\usepackage{appendix}
\usepackage[symbol]{footmisc}
\usepackage{a4wide}
\usepackage{graphicx}
 \usepackage[all]{xy}
 \usepackage{enumerate}
\usepackage{xcolor}
\usepackage{boxhandler}
\usepackage{caption}
\usepackage{lipsum}
\usepackage[symbol]{footmisc}

\newcommand{\vertiii}[1]{{\left\vert\kern-0.25ex\left\vert\kern-0.25ex\left\vert #1 
    \right\vert\kern-0.25ex\right\vert\kern-0.25ex\right\vert}}

\newtheorem*{theorem*}{Main Theorem}
\newtheorem*{coro*}{Corollary}
\newtheorem{theorem}{Theorem}

\newtheorem{thm}{Theorem}

\newtheorem{lem}{Lemma}
\newtheorem{prop}[theorem]{Proposition}
\newtheorem{coro}[theorem]{Corollary}
\newtheorem{rem}[theorem]{Remark}
\theoremstyle{definition}

\newcommand{\diam}{\mathop{\mathrm{diam}}}

\theoremstyle{remark}

\numberwithin{equation}{section}

\begin{document}

\title{Maximal measure and entropic continuity of Lyapunov exponents for $\mathcal C^r$  surface diffeomorphisms with large entropy}

\author{David Burguet}
\address{Sorbonne Universite, LPSM, 75005 Paris, France}
  \email{david.burguet@upmc.fr}  

\date{September 2022}


\begin{abstract}We prove a finite smooth version of the entropic continuity of Lyapunov exponents proved recently by Buzzi, Crovisier and Sarig for $\mathcal C^\infty$ surface diffeomorphisms \cite{BCS2}. 
As a consequence we show that any $\mathcal C^r$, $r>1$, smooth surface diffeomorphism $f$  with $h_{top}(f)> \frac{1}{r}\limsup_n\frac{1}{n}\log^+\|df^n\|_\infty$ admits a measure of maximal entropy. We also prove the $\mathcal C^r$ continuity  of the topological  entropy at $f$.
\end{abstract}

\maketitle

\pagestyle{myheadings} \markboth{\normalsize\sc David
Burguet}{\normalsize\sc Existence of maximal measure for $C^r$  surface diffeos}

\section*{Introduction}

The entropy of a dynamical system quantifies the dynamical complexity by counting distinct orbits. 
There are topological and measure theoretical versions which are related by a variational principle  : the topological entropy of a continuous map on a compact space is equal to the supremum of the entropy of the invariant (probability) measures. An invariant measure is said to be of maximal entropy (or a maximal measure) when its entropy is equal to the topological entropy, i.e. this measure realizes the supremum in the variational principle. In general a topological system may not admit  a measure of maximal entropy. But such a measure exists for dynamical systems satisfying some expansiveness properties. In particular 
Newhouse  \cite{new} has proved their existence for $C^{\infty}$ systems by using Yomdin's theory. In the present paper we show the existence of a measure of maximal entropy for $\mathcal C^r$, $1<r<+\infty$, smooth surface diffeomorphisms  with large entropy. 

Other important dynamical quantities for smooth systems are given by the Lyapunov exponents which estimate the exponential growth of the derivative. For $\mathcal C^\infty$ surface diffeomorphisms, J. Buzzi, S. Crovisier and O. Sarig proved recently a property of continuity  in the entropy of the Lyapunov exponents with many statistical applications \cite{BCS2}. More precisely, they showed that for a $\mathcal C^\infty$ surface diffeomorphism $f$, if $\nu_k$ is a converging sequence of ergodic measures  with $\lim_k h(\nu_k)=h_{top}(f)$, then the Lyapunov exponents of $\nu_k$ are going to the (average) Lyapunov exponents of the limit (which is a measure of maximal entropy).  We prove a  $\mathcal C^r $ version of this fact for $1<r<+\infty$.

\section{Statements}

We define now some notations to state our main results. Fix  a compact Riemannian surface $(\mathbf M, \|\cdot\|)$. For $r>1$ we let $\mathrm{Diff}^r(\mathbf M)$ be the set of $\mathcal C^r$  diffeomorphisms of $\mathbf M$. For $f\in \mathrm{Diff}^r(\mathbf M)$  we let $F:\mathbb PT\mathbf M\circlearrowleft$  be the induced map on the projective tangent bundle $\mathbb PT\mathbf M=T^1\mathbf M/{\pm 1}$  and we denote by $\phi, \psi :\mathbb PT\mathbf M\rightarrow \mathbb R$ the continuous observables on $\mathbb PT\mathbf M$ given respectively  by $\phi:(x,v)\mapsto \log \|d_xf(v)\|$ and $\psi:(x,v)\mapsto \log \|d_xf(v)\|-\frac{1}{r}\log^+\|d_xf\|$ with $\|d_xf\|=\sup_{v\in T_x\mathbf M\setminus \{0\}}\frac{\|d_xf(v)\|}{\|v\|}$.
 For $k\in \mathbb N^*$ we define more generally $\phi_k:(x,v)\mapsto \log \|d_xf^k(v)\|$ and $\psi_k:(x,v)\mapsto\phi_k(x,v)-\frac{1}{r}\sum_{l=0}^{k-1}\log^+\|d_{f^kx}f\|$. 
 Then we let $\lambda^{+}(x)$ and $\lambda^{-}(x)$ be the pointwise Lyapunov exponents given by  $\lambda^{+}(x)= \limsup_{n\rightarrow +\infty}\frac{1}{n}\log \|d_xf^n\|$ and  $\lambda^{-}(x)=\liminf_{n\rightarrow -\infty}\frac{1}{n}\log \|d_xf^n\|$ for any $x\in \mathbf M$ and $\lambda^+(\mu)=\int \lambda^+(x) \, d\mu(x)$, $\lambda^-(\mu)=\int \lambda^-(x) \, d\mu(x)$,  for any $f$-invariant  measure $\mu$.% and $\lambda^+(\hat \mu)=\int \phi\, d\hat\mu$ for any $F$-invariant measure $\hat\mu$. 

 Also we put  $\lambda^+(f):=\lim_n\frac{1}{n}\log^+ \|df^n\|_\infty$ with $\|df^n\|_\infty=\sup_{x\in \mathbf M}\|d_xf^n\|$. The function  $f\mapsto \lambda^+(f)$ is upper semi-continuous in the $\mathcal C^1$ topology on the set of $\mathcal C^1$ diffeomorphisms on $\mathbf M$. For an $f$-invariant measure $\mu$ with $\lambda^+(x)>0\geq \lambda^-(x)$ for $\mu$ a.e. $x$, there are  by Oseledets\footnote[4]{We refer to \cite{Pes} for background on Lyapunov exponents and Pesin theory.} theorem one-dimensional  invariant vector spaces $\mathcal{E}_+(x)$ and $\mathcal{E}_-(x)$, resp. called the unstable and stable Oseledets bundle,  such that $$\forall \,  \mu \text{ a.e. } x\ \forall v\in \mathcal{E}_\pm(x)\setminus \{0\}, \  \lim_{n\rightarrow \pm \infty}\frac{1}{n}\log \|d_xf^n(v)\|=\lambda^{\pm}(x).$$
 Then we let  $\hat \mu^+$ be the $F$-invariant measure given by the lift of $\mu$ on $\mathbb PT \mathbf M$ with $\hat \mu^+(\mathcal E_+)=1$. When writing $\hat \mu^+$ we assume implicitly that  the push-forward  measure $\mu$ on $\mathbf M$  satisfies $\lambda^+(x)>0\geq \lambda^-(x)$ for $\mu$ a.e. $x$.

A sequence of $\mathcal C^r$, with $r>1$, surface diffeomorphisms $(f_k)_k$ on $\mathbf M$ is said to converge $\mathcal C^r$ weakly to a diffeomorphism $f$, when $f_k$ goes to $f$ in the $\mathcal C^1$ topology and the sequence $(f_k)_k$ is $\mathcal C^r$ bounded. In particular $f$ is $\mathcal C^{r-1}$.

\begin{thm}[Buzzi-Crovisier-Sarig, Theorem C   \cite{BCS2}]\label{cochon}
Let $(f_k)_{k\in \mathbb N}$ be a sequence of $\mathcal C^r$, with $r>1$,  surface diffeomorphisms converging $\mathcal C^r$  weakly to a diffeomorphism $f$.  Let $(F_k)_{k\in \mathbb N}$ and $F$ be the lifts of $(f_k)_{k\in \mathbb N}$ and $f$ to $\mathbb P T\mathbf M$. Assume there is a   sequence $(\hat \nu_k^+)_k$ of ergodic $F_k$-invariant measures  converging to $\hat \mu$.\\

  Then there are  $\beta\in [0,1]$ and $F$-invariant  measures $\hat \mu_0$ and $\hat \mu_1^+$ with $\hat \mu= (1-\beta)\hat \mu_0+\beta\hat\mu_1^+$, such that:
$$\limsup_{k\rightarrow +\infty} h(\nu_k)\leq \beta h(\mu_1)+\frac{\lambda^+(f)+\lambda^+(f^{-1})}{r-1}.$$
\end{thm}

%Observe that $\lambda^+(\nu_k)=\int \phi\, d\hat \nu_k^+\xrightarrow{k}\int \phi\, d\hat \mu=\beta\lambda^+(\mu_1)$, therefore $\beta =\lim_k\frac{\lambda^+(\nu_k)}{\lambda^+(\mu_1)}$.
 In particular when $f$ ($=f_k$ for all $k$) is $\mathcal C^\infty$ and $h(\nu_k)$ goes to the topological entropy of $f$, then $\beta$ is equal to $1$ and therefore $\lambda^+(\nu_k)$ goes to $\lambda^+(\mu)$: 
 \begin{coro*}[Entropic continuity of Lyapunov exponents \cite{BCS2}]\label{fir}
 Let $f$ be a $\mathcal C^\infty$ surface diffeomorphism with  $h_{top}(f)>0$.\\

 Then  if   $(\nu_k)_k$ is a sequence of ergodic measures  converging to $\mu$ with   $\lim_k h(\nu_k)=h_{top}(f)$, then \begin{itemize}
\item $h(\mu)=h_{top}(f)$ \footnote[4]{This follows from the upper semi-continuity of the entropy function $h$ on the set of $f$-invariant probability measures for a $\mathcal C^\infty$ diffeomorphism $f$ (in any dimension), which was first proved by Newhouse in \cite{new}.},   
\item $\lim_k\lambda^+(\nu_k)=\lambda^+(\mu)$.
\end{itemize}
\end{coro*}

 We state an improved version  of Buzzi-Crovisier-Sarig Theorem, which allows to prove the same entropy continuity of Lyapunov exponents for $\mathcal C^r$, $1<r<+\infty$, surface diffeomorphisms with large enough entropy (see Corollary \ref{fir}).

\begin{theorem*}\label{ense}
Let $(f_k)_{k\in \mathbb N}$ be a sequence of $\mathcal C^r$, with $r>1$,  surface diffeomorphisms converging $\mathcal C^r$  weakly to a diffeomorphism $f$.   Let $(F_k)_{k\in \mathbb N}$ and $F$ be the lifts of $(f_k)_{k\in \mathbb N}$ and $f$ to $\mathbb P T\mathbf M$. Assume there is a   sequence $(\hat \nu_k^+)_k$ of ergodic $F_k$-invariant measures  converging to $\hat \mu$.\\

  Then for any $\alpha>\frac{\lambda^+(f)}{r}$, there are    $\beta=\beta_\alpha\in [0,1]$ and  $F$-invariant  measures $\hat \mu_0=\hat\mu_{0,\alpha}$ and $\hat \mu_1^+=\hat \mu_{1,\alpha}^+$ with $\hat \mu= (1-\beta)\hat \mu_0+\beta\hat\mu_1^+$,     such that:
$$\limsup_{k\rightarrow +\infty} h(\nu_k)\leq \beta h(\mu_1)+(1-\beta)\alpha.$$
\end{theorem*}

In the appendix we explain how the Main Theorem implies Buzzi-Crovisier-Sarig statement. We state now some consequences of the  Main Theorem. 

\begin{coro}[Existence of maximal measures and entropic continuity of Lyapunov exponents]\label{fir}

 Let $f$ be a $\mathcal C^r$, with $r>1$,  surface diffeomorphism satisfying $h_{top}(f)> \frac{\lambda^+(f)}{r}$.\\

 Then $f$ admits a measure of maximal entropy. More precisely, if   $(\nu_k)_k$ is a sequence of ergodic measures  converging to $\mu$ with   $\lim_k h(\nu_k)=h_{top}(f)$, then \begin{itemize}
\item $h(\mu)=h_{top}(f)$,  
\item $\lim_k\lambda^+(\nu_k)=\lambda^+(\mu)$.
\end{itemize}
\end{coro}

%A $\mathcal \mathcal C^r$, with $r>1$,  surface diffeomorphism satisyfing $ h_{top}(f)> \frac{\lambda^+(f)}{r}$ is said to have \textit{large} entropy.
 It was proved in \cite{BCS1} that any $\mathcal C^r$ surface diffeomorphism satisfying $h_{top}(f)> \frac{\lambda^+(f)}{r}$ admits  at most finitely many ergodic  measures of maximal entropy. On the other hand, J. Buzzi has built examples of $\mathcal C^r$  surface diffeomorphisms for any $+\infty>r>1$ with $\frac{h_{top}(f)}{\lambda^+(f)}$ arbitrarily close to $1/r$  without a measure of maximal entropy \cite{buz}.  It is expected that for any $r>1$ there are $\mathcal C^r$ surface diffeomorphisms satisfying $h_{top}(f)=\frac{\lambda^+(f)}{r}>0$ without measure of maximal entropy or with infinitely many such ergodic measures, but these questions are still open.   Such results were already known for interval maps \cite{bbur,buru,buzthe}.

\begin{proof}
We consider  the constant sequence of diffeomorphisms equal to $f$. By taking a subsequence, we can assume that $(\hat\nu_k^+)_k$ is converging to a lift $\hat \mu$ of $\mu$.  
By using the notations of the Main Theorem  with $h_{top}(f)>\alpha>\frac{\lambda^+(f)}{r}$, we have 
 \begin{align*}
 h_{top}(f)&= \lim_{k\rightarrow +\infty} h(\nu_k),\\
 & \leq \beta h(\mu_1)+(1-\beta)\alpha, \\
 &\leq \beta h_{top}(f)+(1-\beta)\alpha,\\
 (1- \beta) h_{top}(f)&\leq (1-\beta)\alpha.
 \end{align*}
But  $h_{top}(f)> \alpha$, therefore $\beta=1$, i.e. $\hat \mu_1^+=\hat \mu$ and  $\lim_k\lambda^+(\nu_k)=\lambda^+(\mu)$. Moreover $h_{top}(f)=\lim_{k\rightarrow +\infty} h(\nu_k)\leq \beta h(\mu_1)+(1-\beta)\alpha=h(\mu)$. Consequently $\mu$ is a measure of maximal entropy  of $f$.

\end{proof}

\begin{coro}[Continuity of topological entropy and maximal measures]
Let $(f_k)_k$  be a sequence of  $\mathcal C^r$, with  $r>1$, surface diffeomorphisms converging $\mathcal C^r$ weakly to a diffeomorphism $f$  with 
$h_{top}(f)\geq \frac{\lambda^+(f)}{r}$.\\

 Then 
$$h_{top}(f)=\lim_k h_{top}(f_k).$$

 Moreover if $h_{top}(f)>\frac{\lambda^+(f)}{r}$ and $\nu_k$ is a maximal measure of $f_k$ for large  $k$, then any limit measure of $(\nu_k)_k$ for the weak-$*$ topology is a maximal measure  of $f$. 
\end{coro}

\begin{proof}

By Katok's horseshoes theorem \cite{Kat}, the topological entropy is lower semi-continuous for the $\mathcal C^1$ topology on the set of $\mathcal C^r$ surface diffeomorphisms. Therefore it is enough to show the upper semi-continuity. 

By the variational principle there is a sequence   of probability measures $(\nu_k)_{k\in K}$, $K\subset \mathbb N$ with $\sharp K=\infty$, such that : 
\begin{itemize}
\item $\nu_k$ is an ergodic $f_k$-invariant measure for each $k$,
\item $\lim_{k\in K}h(\nu_k)=\limsup_{k\in \mathbb N}h_{top}(f_k)$.
\end{itemize}

By extracting a subsequence we can assume $(\hat \nu_k^+)_k$ is converging to a $F$-invariant measure $\hat\mu$ in the weak-$*$ topology. We can then apply the Main Theorem for any $\alpha> \frac{\lambda^+(f)}{r}$ to get for some $f$-invariant measures $\mu_1, \mu_0$ and $\beta \in [0,1]$ (depending on $\alpha$) with $\mu=(1-\beta)\mu_0+\beta\mu_1$:
\begin{align}\label{grave}
 \limsup_kh_{top}(f_k)&= \lim_{k} h(\nu_k), \nonumber\\
&\leq \beta h(\mu_1) +(1-\beta)\alpha,\\
&\leq \beta h_{top}(f)+(1-\beta)\alpha, \nonumber\\
&\leq \max(h_{top}(f), \alpha). \nonumber
\end{align}
By letting  $\alpha$ go to $\frac{\lambda^+(f)}{r}$ we get 
\begin{align*}
\limsup_kh_{top}(f_k)&\leq h_{top}(f).
\end{align*}
If $h_{top}(f)>\frac{\lambda^+(f)}{r}$, we can fix $\alpha\in \left]\frac{\lambda^+(f)}{r}, h_{top}(f)\right[$ and the inequalities (\ref{grave}) may be then rewritten as follows :
\begin{align*}\limsup_kh_{top}(f_k)&\leq  \beta h(\mu_1) +(1-\beta)\alpha,\\
&\leq h_{top}(f).
\end{align*}
By the lower semi-continuity of the topological entropy, we have $h_{top}(f)\leq \limsup_kh_{top}(f_k)$ and therefore these inequalities are equalities, which implies  $\beta=1$, then $\mu_1=\mu$, and $h(\mu)=h_{top}(f)$.  
\end{proof}

The corresponding result was proved for interval maps in \cite{BurF} by using a different method. We also refer to \cite{BurF} for counterexamples of the upper semi-continuity property for interval maps $f$ with $h_{top}(f)<\frac{\lambda^+(f)}{r}$. Finally, in \cite{buz}, the author built, for any $r>1$, a $\mathcal C^r$ surface diffeomorphism $f$ with $\limsup_{g\xrightarrow{\mathcal C^r}f}h_{top}(g)=\frac{\lambda^+(f)}{r}>h_{top}(f)=0$. We recall also that upper semi-continuity of the topological entropy in the $\mathcal C^\infty$ topology was established in any dimension by Y. Yomdin in \cite{Yom}.

Newhouse proved that for a $\mathcal C^\infty$  system $(\mathbf M,f)$, the entropy function $h:\mathcal M(\mathbf M,f)\rightarrow \mathbb R^+$ is an upper semi-continuous function on the set 
$ \mathcal M(\mathbf M,f)$ of $f$-invariant probability measure. It follows from our Main Thereom, that   the entropy function is upper semi-continuous at ergodic measures with entropy larger than $\frac{\lambda^+(f)}{r}$ for a  $\mathcal C^r$, $r>1$, surface diffeomorphism $f$.

\begin{coro}[Upper semi-continuity of the entropy function at ergodic measures with large entropy]
Let $f:\mathbf M\circlearrowleft $ be a $\mathcal C^r$, $r>1$, surface diffeomorphism. \\
 Then for any  ergodic measure $\mu$ with $h(\mu)\geq\frac{\lambda^+(f)}{r}$, we have 
 $$\limsup_{\nu\rightarrow \mu}h(\nu)\leq h(\mu).$$
\end{coro}
\begin{proof}
By continuity of the ergodic decomposition at ergodic measures and by harmonicity of the entropy function, we have  for any ergodic measure $\mu$ (see e.g. Lemma 8.2.13 in \cite{dow}): 
$$\limsup_{\nu \text{ ergodic}, \, \nu\rightarrow \mu}h(\nu)=\limsup_{\nu\rightarrow \mu}h(\mu).$$
Let $(\nu_k)_{k\in \mathbb N}$ be a sequence of ergodic  $f$-invariant measures with $\lim_{k}h(\nu_k)=\limsup_{\nu\rightarrow \mu}h(\nu)$. 
By extracting a subsequence we can  assume that the sequence $(\hat\nu_k^+)_k$ is converging to some lift $\hat \mu$ of $\mu$. Take $\alpha $ with $\alpha > \frac{\lambda^+(f)}{r}$. Then, in the decomposition  $\hat \mu=(1-\beta)\hat \mu_0+\beta \hat \mu_1^+$ given by the Main Theorem, we have  $\mu_1=\mu_0=\mu$ by ergodicity of $\mu$. Therefore
\begin{align*}
\lim_k h(\nu_k)&\leq \beta h(\mu)+(1-\beta)\alpha,\\
&\leq \max\left( h(\mu), \alpha\right).
\end{align*}
By letting $\alpha$ go to $\frac{\lambda^+(f)}{r}$ we get 
\begin{align*}
\lim_k h(\nu_k)&\leq h(\mu).
\end{align*}
\end{proof}

\section{Main steps of the proof}
We follow the strategy of the proof  of  \cite{BCS2}. We  point  out below  the main differences:
\begin{itemize}
%\item \textit{Entropy along curves.}  By using standard Pesin's theory we relate the entropy of an ergodic invariant measure $\nu_k$ to the "entropy" of the (a priori non-invariant) push-forward $\xi_k$ on a curve $D$ transverse to $\mathcal{E}_-$ along the Pesin stable foliation. In \cite{BCS2} the authors use an anologous result due to Y. Zhang involving Ledrappier-Young theory.
\item \textit{Geometric and neutral empirical component.} For $\lambda^+(\nu_k)>\frac{\lambda^+(f)}{r}$ we split the orbit of a $\nu_k$-typical point $x$ into two parts. We  consider the  empirical measures from $x$ at times lying between to $M$-close consecutive times where the unstable manifold has a "bounded geometry". We take  their limit in $k$, then in $M$. In this way we get an invariant component of $\hat \mu$.  In \cite{BCS2} the authors consider rather such empirical measures for $\alpha$-hyperbolic times and then take the limit when $\alpha$ go to zero.
\item \textit{Entropy computations.} To compute the asymptotic entropy of the $\nu_k$'s, we  use the static entropy w.r.t. partitions and its conditional version. Instead the authors in \cite{BCS2} used Katok's like formulas. 
\item \textit{$\mathcal C^r$ Reparametrizations}. Finally  we use here reparametrization methods from \cite{burens} and \cite{bur} respectively rather than  Yomdin's reparametrizations of the projective action $F$   as done in \cite{BCS2}. This is the principal  difference with \cite{BCS2}. 
\end{itemize}

\subsection{Empirical measures}
Let $(X,T)$ be an invertible topological system, i.e. $T:X\circlearrowleft$ is a homeomorphism of a compact metric space. For a fixed Borel measurable subset $G$ of $X$ we let   $E(x)=E_G(x)$  be the set of times of visits in $G$ from $x\in X$:  $$E(x)=\left\{n \in \mathbb Z,\ T^n x\in G\right\}.$$

 When $a<b$ are two consecutive times in $E(x)$, then $[a,b[$ is called a \textit{neutral block} (by following the terminology of \cite{BCS1}). 
 For all $M\in \mathbb N^*$ we  let  then
\begin{align*}
E^M( x)&=\bigcup_{a<b\in E(x), \ |a-b|\leq M}[a,b[.
\end{align*}
By convention we let $E^\infty(x)=\mathbb Z$.
For $M\in \mathbb N^*$ the complement of $E^M( x)$ is made of disjoint neutral blocks of length larger than $M$.  We consider the associated empirical measures :
$$\forall n, \ \mu_{ x,n}^{M}=\frac{1}{n}\sum_{ k\in E^{M}( x)\cap [0,n[}\delta_{T^k x}.$$
We denote by $\chi^M$  the indicator function of $\{ x, 0\in E^M(x)\}$. The following lemma follows  straightforwardly from  Birkhoff ergodic theorem:
\begin{lem}\label{empiri}
With the above notations, for any $T$-invariant ergodic measure $\nu$,  there is  a set $\mathtt G$ of full $\nu$-measure such that the empirical measures $\left(
\mu_{ x,n}^{M}\right)_n$ are converging for any $x\in \mathtt G$ and any $M\in\mathbb N^*\cup \{\infty\}$ to $
\chi^M \nu$ in the weak-$*$ topology, when $n$ goes to $+\infty$. 
\end{lem}
Fix some $T$-invariant ergodic measure $\nu$.   We  let $ \xi^M=\chi^M \nu$ and  $ \eta^{M}=\nu- \xi^{M}$. 
Moreover we put $\beta_{M}=\int \chi^{M}\, d \nu$, then   $\xi^{M}=\beta_{M}\cdot \underline{\xi}^{M}$ when $\beta_M\neq 0$  and $\eta^{M}=(1-\beta_{M})\cdot \underline \eta^{M}$ when $\beta_M\neq 1$ with $\underline{\xi}^{M}$, $\underline{\eta}^{M}$ being thus  probability measures. % Finally in the next sections we will write  $\xi^{M}$, $\eta^M$, $\mu_{x,n}^{M}$,  ... for the push-forward measures on $M$ of $\hat \xi^{M}$, $\hat\eta^M$, $\hat\mu_{ x,n}^{M}$,... 
Following partially \cite{BCS2}, the measures $\xi^{M}$ and $\eta^M$ are respectively called here the \textit{geometric and neutral components} of $\nu$. In general these measures are not $T$-invariant, but 
$\mathfrak d(\xi^M, T_*\xi^M)\leq 1/M$ for some standard distance $\mathfrak d$ on the set $\mathcal M(X)$ of Borel probability measures on $X$.  From the definition one easily  checks that $\xi^{M}\geq \xi^{N}$ for $M\geq N$.  If $\nu(G)=0$, then  for $\nu$-almost every $x$ we have $\mu_{x,n}^M=0$ for all $n$ and $M$. Assume $G$ has positive  $\nu$-measure. Then, when $M$ goes to infinity,  the function $\chi^M$ goes to $\chi^\infty=1$  almost surely with respect to $\nu$, therefore $\xi^M$ goes to $\nu$. However in general this convergence is not uniform in $\nu$. In the following we consider a sequence $(\nu_k)_k$  of ergodic $T$-invariant measures converging to $\mu$. Then, by a diagonal argument,   we may assume by extracting a  subsequence that $\xi_k^{M}:=\chi^M\nu_k$ is converging for any $M$, when $k$ goes to infinity, to some $\overline{\mu}^M$, which is a priori distinct from  $\chi^M \mu$. We still have $\overline{\mu}^{M}\geq \overline{\mu}^N$ for $M\geq N$, but the limit $\mu_1=\lim_M\overline{\mu}^M$ is a $T$-invariant component of $\mu$, which may differ from  $\mu$.

The next lemma follows from Lemma \ref{empiri} and standard arguments of measure theory:
\begin{lem}\label{neww} There is a Borel subset $\mathtt{H}$ with $\nu(\mathtt H)>\frac{1}{2}$ such that  for any  $M\in \mathbb N$ and  for any continuous function $\varphi:X\rightarrow \mathbb R$: \begin{equation} \frac{1}{n}\sum_{k\in E^M( x)\cap [1,n[}\varphi(T^kx)\xrightarrow{n}\int \varphi\, d \xi^M \text{ uniformly in }x\in \mathtt H.\end{equation}
\end{lem}
\begin{proof} We consider   a  dense countable family  $\mathcal F=(\varphi_k)_{k\in\mathbb N}$  in  the set $\mathcal C^0(X,\mathbb R)$ of real continuous functions  on $X$ endowed with the supremum norm $\|\cdot \|_\infty$. Let $\mathtt G$ be as in Lemma \ref{empiri}. Then for all $k,M$, by Egorov's theorem applied to the pointwise converging sequence $(f_n:\mathtt G\rightarrow \mathbb R)_n=\left(x\mapsto\int \varphi_k\, d\mu_{ x,n}^M\right)_n$, there is a 
subset $\mathtt F_k^M$ of $\mathtt F$ with $\nu(\mathtt F_k^M)>1-\frac{1}{2^{k+M+3}}$ such that 
$\int \varphi_k\, d\mu_{ x,n}^M$ converges to $\int \varphi_k\, d\xi^M$ uniformly in $x\in \mathtt F_k^M$. Let $\mathtt H=\bigcap_{k,M}\mathtt F_k^M$. We have $\nu(\mathtt H)> \frac{1}{2}$.  Then, if $\varphi \in  \mathcal C^0(X,\mathbb R)$, we may find for any $\epsilon>0$ a function $\varphi_k\in \mathcal F$ with $\|\varphi-\varphi_k\|_\infty<\epsilon$.  Let    $M\in \mathbb N$. 
Take $N=N_\epsilon^{k,M}$ such that $|\int \varphi_k\, d\mu_{ x,n}^M-\int \varphi_k\, d\xi^M|<\epsilon$  for $n>N$ and  for  all $x\in \mathtt F_k^M$. In particular for all $x\in \mathtt H$ we have for $n>N$
\begin{align*}
\left|\int \varphi\, d\mu_{\hat x,n}^M-\int \varphi\, d\xi^M\right| \leq &  \left| \int \varphi_k\, d\mu_{x,n}^M- \int \varphi\, d\mu_{  x,n}^M\right|+ \left|\int \varphi_k\, d\mu_{ x,n}^M-\int \varphi_k\, d\xi^M\right|\\& +\left|\int \varphi_k\, d\xi^M-\int \varphi\, d\xi^M\right|,\\ \leq & 2\|\varphi-\varphi_k\|_\infty + \left|\int \varphi_k\, d\mu_{ x,n}^M-\int \varphi_k\, d\xi^M\right|,\\
< &3\epsilon. 
\end{align*}\end{proof}

\subsection{Pesin unstable manifolds}\label{pesi}We consider a smooth compact riemannian manifold $(\mathbf M, \|\cdot\|)$.  
Let $\exp_x$ be the exponential map at $x$ and let $R_{inj}$ be the radius of injectivity of $(\mathbf M, \|\cdot\|)$.     We consider  the distance  $\mathrm d$ on $\mathbf M$ induced by the Riemannian structure. 
Let $f:\mathbf M\circlearrowleft$ be a $\mathcal C^r$, $r>1$, surface diffeomorphism. We denote by $\mathcal R$ the set of Lyapunov regular points with $\lambda^+(x)>0>\lambda^-(x)$.  For $x\in \mathbf M $ we let $W^u(x)$ denote the unstable manifold at $x$ :
$$W^u(x):=\left\{y\in \mathbf M, \ \lim_n\frac{1}{n}\log \mathrm d(f^nx,f^ny)<0 \right\}.$$ By Pesin unstable manifold theorem,  the set $W^u(x)$ for $x\in \mathcal R$ is a $\mathcal C^r$ submanifold tangent to $\mathcal E_+(x)$ at $x$.

For $x\in \mathcal R$, we let $\hat x$ be the vector in $\mathbb PT\mathbf M$ associated to the unstable Oseledets bundle $\mathcal E_+(x)$. For $\delta>0$ the point $x$ is called \textit{$\delta$-hyperbolic} with respect to $\phi$ (resp. $\psi$) when we have $\phi_l(F^{-l}\hat x)\geq \delta l$  (resp. $\psi_l(F^{-l}\hat x)\geq \delta l$)  for all $l>0$.  Note that if $x$ is $\delta$-hyperbolic with respect to $\psi$ then it is $\delta$-hyperbolic with respect to $\phi$.  Let $H_\delta:= \left\{\hat x\in \mathbb PT\mathbf M, \  \forall l>0 \ \psi_l(F^{-l}\hat x)\geq \delta l\right\}$ be the set of $\delta$-hyperbolic points w.r.t. $\psi$. 
\begin{lem}\label{maxi}
Let $\nu$ be an ergodic measure with $ \lambda^+(\nu)-\frac{\log^+ \|df\|_\infty}{r}>\delta>0>\lambda^-(\nu)$.
\\  Then we have $$\hat \nu^+(H_\delta)>0.$$
\end{lem}
\begin{proof}
By applying  the Ergodic Maximal Inequality (see e.g. Theorem 1.1 in \cite{Brown}) to the measure preserving system $(F^{-1},\hat \nu^+)$  with the observable $\psi^\delta=\delta-\psi\circ F^{-1}$, we get
with $A_\delta=\{\hat x\in \mathbb PT\mathbf M, \ \exists k\geq 0 \text{ s.t. } \sum_{l=0}^{k}\psi^\delta(F^{-l}\hat x)>0\}$:
$$\int_{A_\delta} \psi^\delta\, d\hat \nu^+\geq 0.$$
Observe that  $H_\delta=\mathbb PT\mathbf M\setminus A_\delta$. Therefore 
\begin{align*}
\int_{H_\delta}  \psi^\delta\,  d\hat \nu^+&=\int  \psi^\delta \,d\hat\nu^+- \int_{A_\delta}  \psi^\delta\,  d\hat \nu^+,\\
&\leq \int  \psi^\delta \,d\hat\nu^+,\\
& \leq \int (\delta-\psi\circ F^{-1}) \,d\hat\nu^+,\\
&\leq \delta-\lambda^+(\nu)+\frac{1}{r}\int \frac{\log^+\|d_{x}f\|}{r}\, d\nu(x),\\
&<0.
\end{align*}
In particular  we have $\hat \nu^+(H_\delta)>0$.
\end{proof}

A point $x\in \mathcal R$ is said to have $\kappa$-bounded geometry for $\kappa>0$ when $\exp_x^{-1}W^u(x)$ contains the graph of a  \textit{$\kappa$-admissible} map at $x$, which is defined as a $1$-Lipschitz map $f:I\rightarrow \mathcal E_+(x)^{\bot}\subset T_x\mathbf M$,  with $I$ being an interval of $\mathcal E_+(x)$ containing $0$ with length $\kappa$. 
 We let $G_\kappa$ be the subset  of points in  $\mathcal R$   with $\kappa$-bounded geometry. 
\begin{lem}
The set $G_\kappa$ is Borel measurable. 
\end{lem}

\begin{proof}
For $x\in \mathcal R$ we have $W^u(x)=\bigcup_{n\in \mathbb N}f^nW^u_{loc}(f^{-n}x)$ with $W^u_{loc}$ being the Pesin unstable local manifold at $x$. The sequence $\left(f^{n}W^u_{loc}(f^{-n}x)\right)_n$ is increasing in $n$ for the inclusion. Therefore, if we let $G_\kappa^n$ be the subset of points $x$ in $G_\kappa$,
 such that $\exp_x^{-1} f^n W^u_{loc}(f^{-n}x)$ contains the graph of a  $\kappa$-admissible map, then we have 
 $$G_\kappa=\bigcup_nG_\kappa^n.$$
 There are closed subsets, $(\mathcal R_l)_{l\in \mathbb N}$, called the Pesin  blocks,  such that 
$\mathcal R=\bigcup_l\mathcal R_l$ and $x\mapsto W^u_{loc}(x)$ is continuous on $\mathcal R_l$ for each $l$ (see e.g. \cite{Pes}).  Let $(x_p)_p$ be sequence in $G_\kappa^n\cap \mathcal R_l$ which converges to $x\in \mathcal R_l$. By extracting a subsequence  we can assume that the associated sequence of $\kappa$-admissible maps $f_p$ at $x_p$ is converging pointwisely to a $\kappa$-admissible map at $x$, when $p$ goes to infinity. In particular $G_\kappa^n\cap \mathcal R_l$ is a closed set and therefore $G_\kappa=\bigcup_{l,n}\left( G_\kappa^n\cap \mathcal R_l\right) $ is Borel measurable.

\end{proof}

\subsection{Entropy of conditional measures} \label{zeta} We consider an ergodic  hyperbolic  measure  $\nu$, i.e an ergodic measure with $\nu(\mathcal R)=1$. A measurable partition $\varsigma$ is \textit{subordinated} to the Pesin unstable local lamination $W^u_{loc}$  of $\nu$ if the atom $\varsigma(x)$ of $\varsigma$ containing $x$ is a neighborhood of $x$ inside the curve $W^u_{loc}(x)$ and $f^{-1}\varsigma\succ \varsigma$. 
By  Rokhlin's disintegration theorem, there are a measurable set $\mathtt Z$ of full $\nu$-measure  and probability measures $\nu_x$ on $\varsigma(x)$ for $x\in \mathtt  Z$, called the \textit{conditional measures} on  unstable manifolds, satisfying $\nu =\int \nu_x\, d\nu(x)$. Moreover $\nu_y=\nu_x$ for $x,y\in \mathtt Z$ in the same atom of  $\varsigma$. Ledrappier and Strelcyn \cite{LS} have proved the existence of such subordinated measurable partitions. We fix such a subordinated partition $\varsigma$ with respect to $\nu$. For $x\in \mathbf M$, $n\in \mathbb N$ and $\rho>0$, we let   $B_n(x,\rho)$ be  the Bowen ball $B_n(x,\rho):=\bigcap_{0\leq k< n}f^{-k}B(f^kx,\rho)$ (where $B(f^kx,\rho)$ denotes the ball for $\mathrm d$ at $f^kx $ with radius $\rho$).

\begin{lem}\cite{LY}\label{ledr}
For all $\iota>0$, there is $\rho>0$ and a measurable set $\mathtt E\subset  \mathtt Z\cap \mathcal R$     with $\nu(\mathtt E)>\frac{1}{2}$ such that 
\begin{align}\label{leed}\forall x\in \mathtt E, \   \liminf_n-\frac{1}{n}\log \nu_x\left(B_n(x,\rho)\right) \geq h(\nu)-\iota.\end{align}
\end{lem}
The natural projection from $\mathbb PT\mathbf M$ to $\mathbf M$ is denoted by $\pi$.  We  consider a distance  $\hat{\mathrm{d}}$ on the projective tangent bundle $ \mathbb P T\mathbf M$, such that $\hat{\mathrm{d}}(X, Y)\geq \mathrm{d}(\pi X, \pi Y)$  for all $X, Y \in \mathbb P T\mathbf M$. 
We let $\hat \eta^M $ and $\hat \xi^M$ be the neutral and geometric components  of the ergodic $F$-invariant measure $\hat \nu^+$ associated to  $G=H_\delta \cap \pi^{-1}G_\kappa \subset \mathbb PT\mathbf M$, where the parameters $\delta$ and $\kappa$ will be fixed later on independently of $\nu$. The importance of this choice of $G$ will appear in Proposition \ref{paraa} to bound from above  the entropy of the neutral component. We also consider the  projections $\eta^M$ and $\xi^M$  on $\mathbf M$ of $\hat \eta^M $ and $\hat \xi^M$ respectively.  
 By  Lemma \ref{neww} applied to the system $(\mathbb PT\mathbf M, F)$ and to the ergodic measure $\hat \nu^+$, there is a Borel subset $\mathtt H$ of $\mathbb P T\mathbf M$ with $\hat \nu^{+}(\mathtt H)>\frac{1}{2}$ such that  for  any $M\in \mathbb N^*\cup \{\infty\}$ and for any continuous function $\varphi:\mathbb P T\mathbf M\rightarrow \mathbb R$ 
 \begin{equation}\label{unif} \frac{1}{n}\sum_{k\in E^M( \hat x)\cap [1,n[}\varphi(F^k\hat x)\xrightarrow{n}\int \varphi\, d \hat \xi^M \text{ uniformly in }\hat x\in \mathtt H.\end{equation}
%For any $M\in \mathbb N$ we let $\xi^m$ and $\eta^M$ be respectively the hyperbolic and neutral component of $\hat \nu^+$ for the system $(\mathbb PT\mathbf M, F)$.  There is $\rho>0$ and a measurable set $\mathtt F$ with $\nu(\mathtt F)>0$
% such that the empirical measures $\left(\mu_{\hat x,n}^{M}\right)_n$ are converging uniformly in $ x\in \mathtt F$  to $\hat \xi^M$, i.e.  for all continuous function $\psi:\mathbb PT\mathbf M\rightarrow \mathbb R$  \begin{equation}\label{unif} \frac{1}{n}\sum_{k\in E^M(\hat x)\cap [1,n[}\psi(F^k\hat x)\xrightarrow{n}\int \psi\, d\xi^M \text{ uniformly in }x\in \mathtt F\end{equation}and
Fix an error term $\iota>0$ depending\footnote[4]{ In the proof of the Main Theorem we will take $\iota=\iota(\nu_k)\xrightarrow{k}0$ for the converging sequence of ergodic measures $(\nu_k)_k$. } on $\nu$ and let $\rho$ and $\mathtt E$ be as in Lemma \ref{ledr}. Let $\mathtt F= \mathtt E\cap \pi(\mathtt H)$. Note that $\nu(\mathtt F)>0$.  We fix also $x_*\in \mathtt F$ with $\nu_{x_*}(\mathtt F)>0$ and we let $\zeta=\frac{\nu_{x_*}(\cdot)}{\nu_{x_*}(\mathtt F)}$ be the probability measure induced by 
$\nu_{x_*}$ on $\mathtt F$. Observe that $\nu_x=\nu_{x_*}$ for $\zeta$ a.e. $x$. We let $D$ be the $\mathcal C^r$ curve given  by the Pesin local unstable manifold $W^u_{loc}(x_*)$ at $x_*$. For a finite measurable partition $P$ and a Borel probability measure $\mu$ we let $H_\mu(P)$ be the static entropy, $H_\mu(P)=-\sum_{A\in P}\mu(A)\log \mu(A)$. Moreover we let $P^n=\bigvee_{k=0}^{n-1}f^{-k}P$ be the $n$-iterated partition, $n\in \mathbb N$.  We also denote by $P^n_x$ the atom of $P^n$ containing  the point $x\in \mathbf M$.

\begin{lem}\label{roh}
For any (finite measurable) partition $P$ with diameter less than $\rho$, we have 
\begin{align}\label{rhoh}\liminf_n\frac{1}{n}H_{\zeta}(P^n)\geq h(\nu)-\iota.\end{align}
\end{lem}

\begin{proof}
\begin{align*}\liminf_n\frac{1}{n}H_{\zeta}(P^n)& =\liminf_n\int -\frac{1}{n}\log \zeta(P^n_x)\, d\zeta(x), \text{ by the definition of }H_\zeta,\\
&\geq \int\liminf_n -\frac{1}{n}\log \zeta(P^n_x)\, d\zeta(x), \text{ by Fatou's Lemma},\\
&\geq \int \liminf_n  -\frac{1}{n}\log \nu_{x_*}(P^n_x)\, d\zeta(x), \text{ by the definition of } \zeta,\\
&\geq \int \liminf_n  -\frac{1}{n}\log \nu_{x}(P^n_x)\, d\zeta(x), \text{ as } \nu_x=\nu_{x_*} \text{ for $\zeta$ a.e. } x,\\
&\geq \int \liminf_n -\frac{1}{n}\log \nu_{x}(B_n(x,\rho))\, d\zeta(x), \text{ as } \diam(P)<\rho,\\
&\geq  h(\nu)-\iota, \text{ by the choice of }\mathtt F\subset \mathtt E \text{ and }(\ref{leed}).
\end{align*}

\end{proof}

\subsection{Entropy splitting of the neutral and the geometric component} 
In this section we split the entropy contribution of the neutral and geometric components $\hat \eta^M $ and $\hat \xi^M$ of the ergodic $F$-invariant measure $\hat \nu^+$ associated to  a fixed Borel set $G$ of $\mathbb PT\mathbf M$.

Recall that $E(\hat x)$ denotes the set of integers $k$ with $F^k\hat x\in G$.  Fix now  $M$. For each $n\in \mathbb N$ and $x\in \mathtt F$ we let $E_n(x)=E(\hat x)\cap [0,n[$ and $E_n^M(x)=E^M( \hat x)\cap [0,n[$. We also let $\mathtt E_n^M$ be the partition of $\mathtt F$ with atoms  $A_E:=\{x\in D, \, E_n^M(x)=E\}$ for $E\subset [0,n[$. Given a partition $Q$ of $\mathbb P T\mathbf M$, we also let $Q^{\mathtt E_n^M}$ be the partition of $\hat{\mathtt F}:=\left\{\hat x, \, x\in\mathtt{F}\cap D\right\}$ finer than $\pi^{-1}\mathtt E_n^M$ with atoms  $\left\{\hat x\in \hat {\mathtt {F}}, \, E_n^M(x)=E \text{ and } \forall k\in E,\ F^k\hat x \in Q_k\right\}$ for $E\subset [0,n[$ and $(Q_k)_{k\in E}\in Q^E$. We let  $\partial Q$ be the boundary of the partition $Q$, which is the union of the boundaries of its atoms. For a measure $\eta$  and a subset $A$ of $\mathbf M$ with $\eta(A)>0$ we denote by $\eta_A=\frac{\eta(A\cap \cdot)}{\eta(A)}$  the induced probability measure on $A$.   Moreover, for two sets $A,B$  we let $A\Delta B$ denote the symmetric difference of $A$ and $B$, i.e. $A\Delta B=(A\setminus B) \cup (B\setminus A)$. Finally, let $H:]0,1[\rightarrow \mathbb R^+$ be the map $t\mapsto -t\log t-\left(1-t\right)\log\left(1-t\right)$. Recall that $\hat \zeta^+$ is the lift of $\zeta$ on $\mathbb PT\mathbf M$ to the unstable Oseledets bundle (with $\zeta$ as in Subsection \ref{zeta}).

\begin{lem}\label{dd}
For any finite partition $P$ with diameter less than $\rho$ and for any finite partition $Q$ and  any $m\in\mathbb N^*$  with $ \hat \xi^M(\partial Q^m)=0$  we have
\begin{equation}\label{refer}h(\nu)\leq \beta_{M}\frac{1}{m}H_{\underline{\hat \xi}^{M}}(Q^m) +\limsup_{n}\frac{1}{n}H_{\hat\zeta^+}(\pi^{-1}P^{n}|Q^{\mathtt E_n^M})+H(2/M)+\frac{12 \log \sharp Q}{M}+\iota.\end{equation} 
\end{lem}

Before the proof of Lemma \ref{dd}, we first recall a technical lemma from \cite{bur}.

\begin{lem}[Lemma 6  in \cite{bur}]\label{invent}
Let $(X,T)$ be a topological system. Let $\mu$ be a Borel probability measure on $X$ and let $E$ be a finite subset of $\mathbb N$. For any finite partition $Q$ of $X$, we have with $\mu^E:=\frac{1}{\sharp E}\sum_{k\in E}T_*^{k}\mu$ and $Q^E:=\bigvee_{k\in E}T^{-k}Q$: 

$$\frac{1}{\sharp E}H_\mu(Q^E)\leq \frac{1}{m}H_{\mu^E}(Q^m)+6m\frac{\sharp (E+1)\Delta E}{\sharp E}\log \sharp Q.$$

\end{lem}

\begin{proof}[Proof of Lemma \ref{dd}]  As the complement of $E_n^M(x)$ is the disjoint union of neutral blocks with length larger than $M$, there are at most $A_n^M=\sum_{k=0}^{[2n/M]+1}{n\choose k }$ possible values for $E_n^M(x)$ so that 
\begin{align*}
\frac{1}{n}H_{\zeta}(P^n)&=\frac{1}{n}H_{ \zeta}(P^{n}|\mathtt E_n^M)+H_\zeta(\mathtt E_n^M),\\
& \leq \frac{1}{n}H_{\zeta}(P^{n}|\mathtt E_n^M)+\log A_n^M,\\
\liminf_n\frac{1}{n}H_{\zeta}(P^n)& \leq \limsup_n\frac{1}{n}H_{\zeta}(P^{n}|\mathtt E_n^M)+H(2/M) \text{ by using Stirling's formula}.
\end{align*}
Moreover  
\begin{align*}
\frac{1}{n}H_{\zeta}(P^{n}| \mathtt E_n^M)&=\frac{1}{n}H_{\hat \zeta^+}(\pi^{-1}P^{n}| \pi^{-1}\mathtt E_n^M),\\
&\leq \frac{1}{n}H_{\hat \zeta^+}(Q^{ \mathtt E_n^M}| \pi^{-1}\mathtt E_n^M)+ \frac{1}{n}H_{\hat \zeta^+}(\pi^{-1}P^{n}|Q^{ \mathtt E_n^M}).
\end{align*}

For $E\subset [0,n[$ we let $ \hat\zeta^+_{E,n}=\frac{n}{\sharp E}\int \mu_{ \hat x,n}^{M}\,d \zeta_{A_E}(x) $, which may be also written as $ \left(\hat \zeta^+_{\pi^{-1}A_E}\right)^E $ by using the notations of Lemma \ref{invent}. 
By Lemma \ref{invent} applied to the system $(\mathbb PT\mathbf M, F)$ and the measures $\mu:= \hat \zeta^+_{\pi^{-1}A_E}$ for $A_E\in \mathtt E_n^M$ we have  for all $n>m\in \mathbb N^*$:
\begin{align*}
H_{\hat \zeta^+}\left(Q^{ \mathtt  E_n^M}|\pi^{-1} \mathtt  E_n^M\right)&=
\sum_E \zeta(A_E)H_{\hat \zeta^+_{\pi^{-1}A_E}}(Q^E),\\
&\leq\sum_E \zeta(A_E)\sharp E\left(\frac{1}{m}H_{ \hat\zeta^+_{E,n}}(Q^m)+6m\frac{\sharp (E+1)\Delta E}{\sharp E}\log \sharp Q\right).
\end{align*}

Recall again that if $E=E_n^M(x)$ for some $x$ then the complement set of $E$ in $[1,n[$ is made of neutral blocks of length larger than $M$, therefore $\sharp (E+1)\Delta E\leq \frac{2M}{n}$.  Moreover it  follows from $\xi^M(\partial Q^m)=0$ and (\ref{unif}), that  $\mu_{ \hat x,n}^{M}(A^m)$ for $A^m\in Q^m$ and $\sharp E_n^M(x)/n$ are converging 
to $\underline{\hat \xi}^{M}(A^m)$ and $\beta_M$ respectively  uniformly in $x\in \mathtt F$ when $n
$ goes to infinity.  Then we get by taking the limit  in 
$n$:
\begin{align*}
\limsup_{n}\frac{1}{n}H_{\hat \zeta^+}\left(Q^{ \mathtt  E_n^M}|\pi^{-1} \mathtt  E_n^M\right)\leq & \beta_{M}\frac{1}{m}H_{\underline{\hat \xi}^{M}}(Q^m)+\frac{12m \log \sharp Q}{M},\\
h(\nu)-\iota\leq  \liminf_n\frac{1}{n}H_{\zeta}(P^n)\leq & \beta_{M}\frac{1}{m}H_{\underline{\hat \xi}^{M}}(Q^m) +\limsup_{n}\frac{1}{n}H_{\hat\zeta^+}(\pi^{-1}P^{n}|Q^{\mathtt E_n^M})\\&+H(2/M)+\frac{12m \log \sharp Q}{M}.
\end{align*}

%,\\& \leq \beta_{M}\frac{1}{m}H_{\zeta^{M}}(Q^m)+\frac{3 \log \sharp Q}{M\overline{\beta}}.

\end{proof}

\subsection{Bounding the entropy  of the neutral component} 

For a $\mathcal C^1$ diffeomorphism $f$ on $\mathbf M$ we put $C(f):=2A_fH(A_f^{-1})+\frac{\log^+ \|df\|_\infty}{r}+B_r$ with $A_f=\log^+ \|df\|_\infty+\log^+\|df^{-1}\|_\infty+1$ and a universal constant $B_r$ depending only $r$ precised later on. Clearly $f\mapsto C(f)$ is continuous in the $\mathcal C^1$ topology and $\frac{\lambda^+(f)}{r}=\lim_{\mathbb N \ni p\rightarrow +\infty}\frac{C(f^p)}{p}$ whenever $\lambda^+(f)>0$ (indeed $A_{f^p}\xrightarrow{p}+\infty$, therefore $H(A_{f^p}^{-1})\xrightarrow{p}0$). In particular, if $\frac{\lambda^+(f)}{r}<\alpha$ and $f_k\xrightarrow{k}f$ in the $\mathcal C^1$ topology, then there is $p$ with $\lim_k\frac{C(f_k^p)}{p}<\alpha$.

In this section we consider the empirical measures associated to an ergodic hyperbolic measure $\nu$ with $\lambda^+(\nu)>\frac{\log \|df\|_\infty}{r}+\delta$, $\delta>0$. Without loss of generality we can assume $\delta<\frac{r-1}{r}\log 2$. Then by Lemma \ref{maxi} we have $\hat\nu^+(H_\delta)>0$. For $ x\in \mathcal R$ we let $m_n( x)=\max\{k < n, \ F^k\hat x\in H_\delta\}$. By a standard application of Birkhoff  ergodic theorem we have 
$$\frac{m_n( x)}{n}\xrightarrow{n}1 \text{ for $ \nu$ a.e. $x$.}$$
 By taking a smaller subset $\mathtt F$, we can assume the above convergence of $m_n$ is uniform on $\mathtt F$ and that $\sup_{x\in \mathtt F}\min\{k \leq n, \ F^k\hat x\in H_\delta\}\leq N$ for some positive integer $N$.
 
   We bound the term $\limsup_{n}\frac{1}{n}H_{\hat \zeta^+}(\pi^{-1}P^{n}|Q^{\mathtt E_n^M})$ in the right hand side of (\ref{refer}) Lemma \ref{dd}, which corresponds to the local entropy contribution plus  the entropy  in the neutral part. %In the next statement, for any $\epsilon>0$ and $p\in \mathbb N$ we let $\|df^p\|_\epsilon:x\mapsto \|d_x  f^p\|_\epsilon:=\sup_{y\in  B(x,\epsilon)}\|d_y f^p\|$. 

\begin{lem}\label{fee}
There is $\kappa>0$ depending only on $\|d^kf\|_\infty$, $2\leq k\leq r$, \footnote[4]{Here $$\|d^kf\|_\infty=\sup_{\alpha\in \mathbb N^{2}, \, |\alpha|=k}\sup_{x,y}\left\|\partial_y^\alpha\left(\exp_{f(x)}^{-1}\circ f\circ \exp_x\right)(\cdot)\right\|_\infty$$ } such that the empirical measures associated to $G:=\pi^{-1}G_\kappa\cap H_\delta$ satisfy the following properties. For all $q, M\in \mathbb N^*$, there are $\epsilon_q>0$  depending only on $\|d^k(f^q)\|_\infty$, $2\leq k\leq r$ and $\gamma_{q,M}(f)>0$ such that for any partition $Q$ of $\mathbb PT\mathbf M$ with diameter less than $\epsilon_{q}$, we have:
\begin{align*}\limsup_n\frac{1}{n}H_{\hat \zeta^+}(\pi^{-1}P^{n}|Q^{ \mathtt  E_n^M})\leq & (1-\beta_{M})C(f)\\& +\left(\log 2+\frac{1}{r-1}\right)\left(\int \frac{\log ^+\|df^q\|}{q}d\xi^{M}-\int \phi\, d\hat \xi^{M}\right)\\& +\gamma_{q,M}(f),
\end{align*}
where the error term  $\gamma_{q,M}(f)$  satisfies  \begin{equation}\label{todd}\forall K>0 \ \limsup_{q}\limsup_M\left(\sup_{f\in \mathrm{Diff}^r(\mathbf M)} \left\{\gamma_{q,M}(f) \ | \ \|df\|_\infty \vee \|df^{-1}\|_\infty<K\right\}\right)=0.\end{equation} 
\end{lem}

The proof of  Lemma \ref{fee} appears after the statement of  Proposition \ref{paraa}, which is  a \textit{semi-local Reparametrization Lemma}.  %Without loss of generality we can assume $\nu(G_\alpha)>0$.

%\begin{lem}\label{shit}For any $0<\alpha'<\alpha$, there exists $C>0$ such that for   any strongly $n$-bounded  subcurve $\sigma$ of $D$ intersecting $\mathtt F$, the  length of $f^k\circ \sigma$ is less than  $Ce^{-\alpha' (n-k)}$.\end{lem}
%\begin{proof}Let $x\in \mathtt F\cap \Ima(\sigma)$. For $0\leq k\leq m_n(x)$, we have $\phi_{m_n(\hat x)-k}(F^{k}\hat x)\geq \alpha(m_n(\hat x)-k)$ as $m_n(\hat x)$ belongs to $E(\hat x)$. Therefore  the length of $f^k\circ \sigma $ is less than $8\epsilon e^{\alpha (k-m_n(x))}$   by the aforementioned bounded distorsion property. Finally,  observe that  $\frac{m_n(\hat x)}{n}$ is converging uniformly to $1$ on $\mathtt F$ by   definition of $\mathtt F$ as we have $\nu(G_\alpha)>0$.
%\end{proof}

\begin{prop}\label{paraa}There is $\kappa>0$ depending only on $\|d^kf\|_\infty$, $2\leq k\leq r$, such that the empirical measures associated to $G:=\pi^{-1}G_\kappa\cap H_\delta$  satisfy the following properties. For all $q,M\in \mathbb N^*$ there are  $\epsilon_q>0$ depending only on   $\|d^k(f^q)\|_\infty$, $2\leq k\leq r$   and $\gamma_{q,M}(f)>0$  satisfying (\ref{todd})   such that for any partition $Q$ with diameter less than $\epsilon<\epsilon_q$, we have  for  $n$ large enough :  \\

Any atom $F_n$ of the partition  $ Q^{\mathtt E_n^M}$ may be covered by a family $\Psi_{F_n}$ of $\mathcal C^r$ curves $\psi:[-1,1]\rightarrow \mathbf M$  satisfying  $\|d(f^k\circ \psi)\|_\infty\leq 1$ for any $k=0,\cdots, n-1$, such that 

 \begin{align*}\frac{1}{n}\log \sharp \Psi_{F_n}\leq &\left(1-\frac{\sharp E_n^M}{n}\right)C(f)\\ &+\left(\log 2+\frac{1}{r-1}\right)\left(\int \frac{\log ^+\|d_xf^q\|_{\epsilon}}{q}\, d\zeta_{F_n}^{M}(x)-\int \phi\, d\hat{\zeta}_{F_n}^{M}\right)\\&+\gamma_{q,M}(f)+\tau_n,
 \end{align*}
where $\lim_n\tau_n=0$, $E_n^M=E_n^M(x)$ for $x\in F_n$,   $\hat{\zeta}_{F_n}^{M}=\int \mu_{\hat x,n}^M\, d\zeta_{F_n}(x)$ and $\zeta_{F_n}^{M}=\pi_*\hat{\zeta}_{F_n}^{M}$ its push-forward on $\mathbf M$.
\end{prop}

The proof of Proposition \ref{paraa} is given in the last section. 
Proposition \ref{paraa} is very similar to the Reparametrization Lemma in \cite{burens}. Here we reparametrize  an atom $F_n$ of $Q^{\mathtt E_n^M}$ instead of $Q^n$ in \cite{burens}.% The arguments are  the same as in \cite{burens} for  the hyperbolic part, whereas in the neutral part we use the bounded geometry property at hyperbolic times established in \cite{bur}. 

\begin{proof}[Proof of Lemma \ref{fee} assuming Proposition \ref{paraa}] We take $\kappa>0$ and $\epsilon_q>0$ as in Proposition \ref{paraa}. Observe that 
$$H_{\hat \zeta^+}(\pi^{-1}P^{n}|Q^{\mathtt E_n^M})\leq \sum_{F_n\in  Q^{\mathtt E_n^M}}\hat\zeta^+(F_n)\log \sharp\{ A^n\in P^n, \ \pi^{-1}(A^n )\cap \hat{\mathtt F}\cap F_n\neq \emptyset\}.$$
 
 As $\nu(\partial P)=0$, for all $\gamma>0$, there is  $\chi>0$ and a continuous function $\vartheta:\mathbf M\rightarrow \mathbb R^+$ 
equal to $1$ on  the $\chi$-neighborhood  $\partial P^\chi$ of $\partial P$ satisfying  $\int \vartheta\,d\nu<\gamma$.  Then, by applying  (\ref{unif})  with $\varphi: \hat x\mapsto \vartheta(x)$ and $M=\infty$,  we have  uniformly in $x\in\mathtt F\subset \pi(\mathtt H)$:  \begin{equation}\label{drd}\limsup_n\frac{1}{n}\sharp\{0\leq k<n, \ f^kx\in  \partial P^\chi\}\leq \lim_n\frac{1}{n}\sum_{k=0}^{n-1}\vartheta(f^kx)=\int \vartheta\, d\nu<\gamma.
\end{equation} 

Assume that for arbitrarily large $n$ there is $F_n\in Q^{\mathtt E_n^M}$ and $\psi\in \Psi_{F_n}$  with $\sharp \{A^n\in P^n, \ A^n\cap \psi([-1,1])\cap \mathtt F\neq \emptyset\}>([\chi^{-1}]+1)\sharp P^{\gamma n}$. As $\|d(f^k\circ \psi)\|_\infty\leq 1$ for $0\leq k<n$  we may reparametrize $\psi$ on $\mathtt F$  by  $[\chi^{-1}]+1$ affine contractions $\theta$ so that the length of $f^k\circ\psi\circ \theta $ is less than $\chi$ for all $0\leq k< n$ and  $(\psi\circ \theta)([-1,1])\cap \mathtt F\neq \emptyset$.   Then we have $\sharp \{0\leq k<n, \ \partial P\cap (f^k\circ\psi\circ \theta)([-1,1])\neq \emptyset\}>\gamma n$ for some $\theta$.  In  particular we get $\sharp \{0\leq k<n, \ f^k x \in \partial P^\chi\}>\gamma n$ for any $x\in \psi\circ \theta([-1,1])$, which contradicts (\ref{drd}). Therefore we have  $$\limsup_{n}\sup_{F_n,\, \psi\in \Psi_{F_n}}\frac{1}{n}\log \left\{A^n\in P^n, \ A^n\cap \psi([-1,1]) \cap \mathtt F\neq \emptyset\right\}=0.$$
Together with Proposition \ref{paraa} and Lemma \ref{neww} we get 
\begin{align*}
\limsup_n\frac{1}{n} H_{\hat\zeta^+}(\pi^{-1}P^{n}|Q^{\mathtt E_n^M})&\leq \limsup_n \sum_{F_n\in  Q^{\mathtt E_n^M}}\hat\zeta^+(F_n)\frac{1}{n}\log \sharp \Psi_{F_n},\\
&\leq \limsup_n  \sum_{F_n\in  Q^{\mathtt E_n^M}}\hat\zeta^+(F_n)\left(1-\frac{\sharp E_n^M}{n}\right)C(f)+\\
&+\limsup_n  \sum_{F_n\in  Q^{\mathtt E_n^M}}\hat\zeta^+(F_n) \left(\log 2+\frac{1}{r-1}\right)\left(\int \frac{\log ^+\|df^q\|}{q}\, d\zeta_{F_n}^{M}-\int \phi\, d\hat{\zeta}_{F_n}^{M}\right) \\ 
& +\gamma_{q,M}(f),\\
&\leq (1-\beta_M)C(f)+\left(\log 2+\frac{1}{r-1}\right)\left(\int \frac{\log ^+\|df^q\|}{q}d\xi^{M}-\int \phi\, d\hat{\xi}^{M}\right)+\gamma_{q,M}(f).
\end{align*}

This concludes the proof of Lemma \ref{fee}.

\end{proof}

By combining  Lemma \ref{fee} and Lemma \ref{dd} we get:
\begin{prop}\label{revi}Let $\kappa$, $\epsilon_q$ and $\gamma_{q,M}(f)$ as in Proposition \ref{paraa}. Then for any $q,M\in \mathbb N^*$ and for any finite partition $Q$ with diameter less than $\epsilon_q$ and    with $ \hat \xi^M(\partial Q^m)=0$  we have  
with $\gamma_{q,Q,M}(f)=\gamma_{q,M}(f)+H\left(\frac{2}{M}\right)+\frac{12 \log \sharp Q}{M}$ : 
\begin{align*}h(\nu)\leq &\beta_{M}\frac{1}{m}H_{\underline{\hat{\xi}}^{M}}(Q^m) +(1-\beta_{M}) C(f)\\ & +\left(\log 2+\frac{1}{r-1}\right)\left(\int \frac{\log ^+\|df^q\|}{q}d\xi^{M}-\int \phi\, d\hat{\xi}^{M}\right)\\ & +\gamma_{q,Q,M}(f)+\iota.
\end{align*} 
\end{prop}

\subsection{Proof of the Main Theorem}

We first reduce the Main Theorem to the following statement.

\begin{prop}\label{reduc}
Let $(f_k)_{k\in \mathbb N}$ be a sequence of $\mathcal C^r$, with $r>1$,  surface diffeomorphisms converging $\mathcal C^r$  weakly to a diffeomorphism $f$.  Assume there is a   sequence $(\hat \nu_k^+)_k$ of ergodic $F_k$-invariant measures  converging to $\hat \mu$ with $\lim_k\lambda^+(\nu_k)>\frac{\log^+ \|df\|_{\infty}}{r}$.\\

  Then, there are   $F$-invariant  measures $\hat \mu_0$ and $\hat \mu_1^+$ with $\hat \mu= (1-\beta)\hat \mu_0+\beta\hat\mu_1^+$, $\beta\in [0,1]$,    such that:
$$\limsup_{k\rightarrow +\infty} h(\nu_k)\leq \beta h(\mu_1)+(1-\beta)C(f).$$
\end{prop}

\begin{proof}[Proof of the Main Theorem assuming Proposition \ref{reduc}]
Let $(\hat\nu_k^+)_k$ be a sequence of ergodic  $F_k$-invariant measures converging to $\hat \mu$.

As previously mentionned, for any $\alpha>\lambda^+(f)/r$ there is $p\in \mathbb N^*$ with $\alpha>\frac{C(f^p)}{p}$. We can also assume $\frac{\log\|df^p\|_\infty }{pr}<\alpha$. Let $\hat\nu_k^{+,p}$ be an ergodic component of $\hat\nu_k^+$ for $F_k^p$ and  let us denote by $\nu_k^p$ its push forward on $\mathbf M$. We have $ h_{f_k^p}(\nu_k^p)=ph_{f_k}(\nu_k)$ for all $k$. By taking a subsequence we can assume that $(\hat\nu_k^{+,p})_k$ is converging.  Its  limit $\hat \mu^p$ satisfies $\frac{1}{p}\sum_{0\leq l<p}F_*^k\hat \mu^p=\hat \mu$.  If  $\lim_k\lambda^+(\nu_k^p)\leq \frac{\log^+ \|df^p\|_{\infty}}{r}<p\alpha$, then by Ruelle's inequality we get \begin{align*}
\limsup_{k\rightarrow +\infty} h_{f_k}(\nu_k)&=\limsup_{k\rightarrow +\infty} \frac{1}{p}h_{f_k^p}(\nu_k^p),\\
&\leq \lim_{k\rightarrow +\infty} \frac{1}{p}\lambda^+(\nu_k^p),\\
&< \alpha.
\end{align*}
 This proves the Main Theorem with $\beta=1$. 

 We consider then the case $\lim_k\lambda^+(\nu_k^p)>\frac{\log^+ \|df^p\|_{\infty}}{r}$.
By applying Proposition 4 to the $p$-power system, we get $F^p$-invariant measure $\hat \mu_0^p$ and $\hat \mu_1^{+,p}$ with $\hat \mu^p= (1-\beta)\hat \mu_0^p+\beta\hat\mu_1^{+,p}$, $\beta\in [0,1]$,    such that we have with $\mu_1^p=\pi_*\hat\mu_1^{+,p}$ :
$$\limsup_{k\rightarrow +\infty} h_{f_k^p}(\nu_k^p)\leq \beta h_{f^p}(\mu_1^p)+(1-\beta)C(f^p).$$
But  $ h_{f^p}(\mu_1^p)=ph_{f}(\mu_1)$ 
 with $\mu_1=\frac{1}{p}\sum_{0\leq l<p}f^k \mu_1^p$. One easily checks that $\hat \mu_1^+=\frac{1}{p}\sum_{0\leq l<p}F^k \hat\mu_1^{+,p}$. Then we have :
\begin{align*}
\limsup_{k\rightarrow +\infty} h_{f_k}(\nu_k)&=\limsup_{k\rightarrow +\infty} \frac{1}{p}h_{f_k^p}(\nu_k^p),\\
&\leq \beta \frac{1}{p}h_{f^p}(\mu_1^p)+(1-\beta)\frac{C(f^p)}{p},\\
&\leq   \beta h_{f}(\mu_1)+(1-\beta)\alpha.
\end{align*}
This concludes the proof of the Main Theorem.
\end{proof}

 % with $\int \phi \,d\hat \mu=\lim_k \lambda^+(\nu_k)>\alpha>\frac{\lambda^+(f)}{r}$.
 
We show now Proposition \ref{reduc} by using Lemma \ref{fee}.  
\begin{proof}[Proof of Proposition \ref{reduc}:] Without loss of generality we can assume $\liminf_kh(\nu_k)>0$. For $\mu$ a.e. $x$, we have $\lambda^-(x)\leq 0$. If not, some ergodic component $\tilde{\mu}$ of $\mu$ would have two positive Lyapunov exponents and therefore should be the periodic measure at  a source $S$ (see e.g. Proposition 4.4 in \cite{pol}).  But then  for large $k$ the probability $\nu_k$ would give positive measure to the  basin of attraction of the sink  $S$ for $f^{-1}$ and therefore $\nu_k$ would be equal to $\tilde{\mu}$ contradicting $\liminf_k h(\nu_k)>0$.

 %There are $F_k^p$-ergodic components  $\hat \nu_k^{+,p}$ of $\hat \nu_k^+$ such that $\hat \nu_k^{+,p}$ is converging when $k$ goes to infinity. Observe that the limit $\hat \mu^p$ satisfies $\hat \mu=\frac{1}{p}\sum_{0\leq l<p}F_k^l\hat \mu^p$ and we have $h_{f_k^p}(\nu_k^{p})=ph_{f_k}(\nu_k)$ with $\nu_k^{p}$ being the push forward on $\mathbf M$ of $\hat \nu_k^{+,p}$. 
Let $\delta>0$ with $\lim_k\lambda^+(\nu_k)> \frac{\log \|df\|_\infty}{r}+\delta$. Then take  $\kappa$ as in  Lemma \ref{fee}. We consider the empirical measures associated to $G=\pi^{-1}G_\kappa\cap H_{\delta}$. 
  By a diagonal argument, there is a subsequence in $k$ such that the geometric component  $\hat \xi_{k}^{M} $ of $\hat \nu_k^+$ is converging to some $\hat\xi_\infty^{M}$ for all  $M\in \mathbb N$.  Let us also denote by $\beta_M^\infty$  the limit in $k$  of $\beta_M^k$. Then consider  a subsequence in $M$ such that $\hat \xi_\infty^{M}$ is converging to $\beta\hat \mu_1$ with $\beta=\lim_M\beta_M^\infty$. We also let $(1-\beta)\hat \mu_0=\hat \mu-\beta\hat\mu_1$. In this way, $\hat\mu_0$ and $\hat\mu_1$ are both probability measures.
\begin{lem}The measures $\hat \mu_0$ and $\hat \mu_1$ satisfy the following properties:
\begin{itemize}
\item $\hat \mu_1$ and $\hat \mu_0$ are $F$-invariant,
\item $\lambda^+(x)\geq \delta$ for $\mu_1$-a.e. $x$ and $\hat \mu_1=\hat \mu_1^+$.
\end{itemize}
\end{lem}
\begin{proof}The neutral blocks in the complement set of $E^M(x)$ have length larger than $M$. Therefore for any continuous function $\varphi:\mathbb PT\mathbf M\rightarrow \mathbb R$ and for any $k$, we have
$$\left|\int  \varphi\, d \hat\xi_k^{M}-\int \varphi\circ F\, d\hat\xi_k^{M}\right|\leq \frac{2\sup_{\hat x} |\varphi(\hat x)|}{M}.$$
Letting $k$, then $M$ go to infinity, we get $\int  \varphi\,  d\hat\mu_1=\int \varphi\circ F \, d\hat\mu_1$, i.e. $\hat \mu_1$ is $F$-invariant. 

We let $K_M$ be the compact subset of  $\mathbb PT\mathbf M$ given by $K_M=\{\hat x \in \mathbb PT\mathbf M, \  \exists 1\leq m\leq M \ \phi_m(\hat x)\geq m\delta\}$. Let $\hat x\in \mathtt  G_k$, where $\mathtt G_k$ is the set where the  empirical measures are converging to  $\hat \xi_k^{M}$ (see Lemma \ref{empiri}). Observe that \begin{equation}\label{wed}\lim_n\mu_{\hat x, n}^M(K_M)=  \hat \xi_k^{M}(K_M)=\hat \xi_k^{M}(\mathbb PT\mathbf M).\end{equation} Indeed  for any $k\in E^M(\hat x)$ there is $1\leq m\leq M$ with $F^m(F^k \hat x)\in G\subset H_\delta$. Moreover, as already mentioned,  $\delta$-hyperbolic points w.r.t. $\psi$ are $\delta$-hyperbolic  w.r.t. $\phi$. Therefore $\phi_m(F^k\hat x)\geq m\delta$. Consequently we have $\lim_n\mu_{\hat x, n}^M(K_M)=\lim_n\mu_{\hat x, n}^M(\mathbb PT\mathbf M)=\hat\xi_k^{M}(\mathbb PT\mathbf M)$. The set $K_M$  being compact in $\mathbb PT\mathbf M$, we get  $\hat\xi_k^{M}(K_M)\geq \lim_n\mu_{\hat x, n}^M(K_M)$ and (\ref{wed}) follows.

Also we have $\hat \xi_\infty^{M}(K_M)\geq \limsup_k \hat \xi_k^{M}(K_M)=\limsup_k \hat \xi_k^{M}(\mathbb PT\mathbf M)=\beta^\infty_M$.  Therefore we have $\hat \mu_1(\bigcup_M K_M)=1$ as $\hat \xi_\infty^{M}$ goes increasingly in $M$ to $\beta\hat \mu_1$.  
The $F$-invariant set $\bigcap_{k\in \mathbb Z}F^{-k}\left(\bigcup_M K_M\right)$  has also  full $\hat \mu_1$-measure and for all $\hat x=(x,v)$ in this set we have $\limsup_n\frac{1}{n}\log \|d_xf^n(v)\|\geq \delta$. 
Consequently the measure $\hat \mu_1$ is supported on the unstable bundle $\mathcal E_+(x)$ and  $\lambda^+(x)\geq \delta$ for $\mu_1$-a.e. $x$. 
\end{proof}

\begin{rem}\label{mieux}
In  Theorem  C of  \cite{BCS2}, the measure $\beta\hat \mu_1^+$ is obtained as the limit when  $\delta$ goes to zero of the  component associated to the set $G^\delta:=\{x, \  \forall l>0 \ \phi_l(\hat x)\geq \delta l \} \supset \pi^{-1}G_{\kappa}\cap H_\delta$. Therefore our measure $\beta_\alpha\hat \mu_{1,\alpha}^+$ is just a component of their measure $\beta\hat \mu_1^+$.
\end{rem}

We pursue now the proof of Proposition \ref{reduc}.  Let $q, M\in \mathbb N^*$. Fix  a sequence $(\iota_k)_k
$ of positive numbers with $\iota_k\xrightarrow{k}0$.  We consider a partition   $Q$ satisfying  $\diam(Q)<\epsilon_q$ with $\epsilon_q$ as in Lemma \ref{fee}. 
The sequence $(f_k)_k$ being  $\mathcal C^r$ bounded, one can choose $\epsilon_q$ independently of $f_k$, $k\in \mathbb N$.

By a standard argument of countability we may  assume that for all $m\in \mathbb N^*$ the boundary of $Q^m$ has zero-measure for $\hat \mu_1^+$ and all the measures $\hat \xi_k^M$, $M\in \mathbb N^*$ and $k \in \mathbb{N}\cup \{\infty\}$.
By applying Proposition \ref{revi}  to $f_k$ and $\nu_k$ we get: 
\begin{align*}h(\nu_k)\leq &\beta^k_{M}\frac{1}{m}H_{\underline{\hat{\xi}_k}^{M}}(Q^m) +(1-\beta^k_{M}) C(f_k)\\ & +\left(\log 2+\frac{1}{r-1}\right)\left(\int \frac{\log ^+\|df_k^q\|}{q}d\xi_k^{M}-\int \phi\, d\hat{\xi_k}^{M}\right)\\ & +\gamma_{q,Q,M}(f_k)+\iota_k.
\end{align*}

By letting $k$, then $M$ go to infinity,  we obtain for all $m$:
\begin{align*}   \limsup_k h(\nu_k)\leq & \beta\frac{1}{m}H_{\hat\mu_1^+}(Q^m)+(1-\beta)C(f)\\  & + \left(\log 2+\frac{1}{r-1}\right)\left( \int \frac{\log ^+\|df^q\|}{q}d\mu_1-\int \phi\, d \hat \mu_1^+ \right)\\& +\limsup_{M}\sup_k\gamma_{q,Q,M}(f_k).
\end{align*}
 By letting $m$ go to infinity,  we get:
\begin{align*}   \limsup_k h(\nu_k)\leq &\beta h(\hat\mu_1^+)+(1-\beta)C(f) \\
& +\left(\log 2+\frac{1}{r-1}\right)\left( \int \frac{\log ^+\|df^q\|}{q}d\mu_1-\int \phi\, d \hat \mu_1^+ \right)\\ & +\limsup_{M}\sup_k\gamma_{q,M}(f_k).\end{align*}

But $h(\hat \mu_1^+)=h(\mu_1)$ as the measure preserving systems associated to $\mu_1$ and $\hat \mu_1^+$ are isomorphic. Moreover we have 
$\int \phi\, d\hat\mu_1^+=\lambda^+(\mu_1)=\lim_{q} \int \frac{\log ^+\|df^q\|}{q}d\mu_1$. Therefore by letting $q$ go to infinity we finally obtain with the asymptotic property (\ref{todd}) of $\gamma_{q,M}$: 
$$\limsup_k h(\nu_k)\leq \beta h(\mu_1)+(1-\beta)C(f).  $$
This concludes the proof of Proposition \ref{reduc}.
\end{proof}

\section{Semi-local  Reparametrization Lemma }

%This last section is devoted to the proof of Lemma \ref{para}.
%In \cite{bur} it was shown that for a $C^\infty$ surface diffeomorphism the geometry was bounded at hyperbolic times. We state a slightly improved version :\begin{lem}\label{cho} There is $\epsilon>0$ and $1\geq\gamma>0$  depending on $\alpha>\frac{\lambda^+(f)}{r}$   such that for any  ample $(\epsilon,r)$-bounded curve $\sigma$ the length of  $D_{n,\sigma}^{r,\epsilon}(x)$ is larger than $\gamma \epsilon$ for any $\alpha$-hyperbolic time $n$ at $x$. \end{lem}Without loss of generality    we can assume $D$  is ample $(\epsilon,r)$-bounded curve with reparametrization $\sigma$. We put $\delta=\gamma\epsilon$ and we consider a partition $Q$ with diameter less than $\delta$. 

In this section  we prove the semi-local \textit{Reparametrization Lemma} stated above in Proposition \ref{paraa}. %It may be deduced from the Reparametrization Lemma in \cite{bur}, but for the sake of completeness we will give  a full proof. %We first recall some terminology from \cite{bur}. 

\subsection{Strongly bounded curves}
To simplify the exposition (by avoiding irrelevant technical details involving the exponential map) we assume that $\mathbf M$ is the two-torus $\mathbb T^2$ with the usual  Riemannian structure inherited from $\mathbb R^2$. Borrowing from \cite{bur}   we first make the following definitions.\\

 %Let $0<R<R_{inj}$ with $\sup_{x\in \mathbf M}\sup_{y\in T_x\mathbf M, \ \|y\| \leq R}\|d_y\exp_x\|<2$. A $\mathcal C^r$ embedded curve $\sigma:[-1,1]\rightarrow \mathbf M$ is said \textit{bounded} when 
%$\sigma([-1, 1])\subset B(x,R)$ with $x=\sigma(0)$ and 
%  $\tilde \sigma
% =\exp_x^{-1}\circ \sigma : [-1,1]\rightarrow T_x\mathbf M$  satisfies
%$\max_{k=2,\cdots,r}\|d^k\tilde\sigma\|_\infty\leq \frac{\|d\tilde\sigma\|_\infty}{6}$.
A $\mathcal C^r$ embedded curve $\sigma:[-1,1]\rightarrow \mathbf M$ is said \textit{bounded} when 
$\max_{k=2,\cdots,r}\|d^k\sigma\|_\infty\leq \frac{\|d\sigma\|_\infty}{6}$.\\

\begin{lem}\label{nonam}
Assume $\sigma $ is a bounded curve. Then for 
any $x\in \sigma([-1,1])$, the curve $\sigma$ contains the graph of a $\kappa$-admissible map  at $x$ with $\kappa= \frac{\|d\sigma\|_\infty}{6}$.
\end{lem}

\begin{proof}Let $x=\sigma(s)$, $s\in [-1,1]$. 
One checks easily  (see Lemma 7 in \cite{burens} for further details) that for all $t\in [-1,1]$ the angle $\angle\sigma'(s), \sigma'(t)<\frac{\pi}{6}\leq 1$ 
and therefore $\int_{0}^1\sigma'(t)\cdot \frac{\sigma'(s)}{\| \sigma'(s)\|} \, dt\geq \frac{\|d\sigma\|_\infty}{6}$. Therefore, as $\sigma'(s)\in \mathcal E_+(x)$,  the image of $\sigma$ contains the graph of an $\frac{\|d\sigma\|_\infty}{6}$-admissible map at $x$. 
\end{proof}

 A $\mathcal C^r$ bounded curve $\sigma:[-1,1]\rightarrow \mathbf M$ is said  \textit{strongly $\epsilon$-bounded }for $\epsilon>0$ if $\|d\sigma\|_\infty\leq \epsilon$.
  %When $\sigma$ is   $n$-bounded, then the distortion is uniformly bounded along $\sigma$, i.e. for all $\hat x=(x,v_x)$, $\hat y=(y,v_y)$ in $\mathbb P TM$  tangent to $\sigma$ and for all $k =0,\cdots,n-1$ we have $\left|\phi_{n-k}(F^k\hat x)-\phi_{n-k}(F^k\hat y)\right|\leq \log 2$  by  Lemma  8 in  \cite{bur}. 
    For $n\in \mathbb N^*$ and $\epsilon>0$ a curve is said  \textit{strongly $(n,\epsilon)$-bounded} when  $f^k\circ \sigma$ is strongly $\epsilon$-bounded for all $k=0,\cdots, n-1$.\\

 We consider a $\mathcal C^r$ smooth  diffeomorphism $g:\mathbf M\circlearrowleft$ with $\mathbb N \ni r\geq 2$.   For $\hat x=(x,v)\in \mathbb PT\mathbf M $ with $\pi(\hat x)=x$, we let $k_g(x)\geq k'_g(\hat x)$ be the following integers:
$$k_g(x):=\left[\log\|d_{x}g \|\right], $$
$$k'_g( \hat x):=\left[\log \|d_xg(v)\|\right]=[\phi_g(\hat x)].$$
In the next lemma, we reparametrize the image by $g$ of a  bounded curve. The proof of this lemma is mostly contained in the proof of the Reparametrization Lemma \cite{bur}, but we reproduce it for the sake of completeness.

\begin{lem}\label{nondyn}
 Let $\frac{R_{inj}}{2}>\epsilon=\epsilon_g>0$
 % such that the image of any $\epsilon$-ball has diameter less than $R_{inj}$. We also assume
 satisfying  $\|d^sg_{2\epsilon}^x\|_\infty\leq 3\epsilon \|d_xg\|$ for all $s=1,\cdots,r$ and all $x\in \mathbf M$, where $g^x_{2\epsilon}= g\circ  \exp_x(2\epsilon\cdot)=g(x+2\epsilon\cdot): \{w_x\in T_x\mathbf M, \ \|w_x\|\leq 1\}\rightarrow \mathbf M$. We assume   $\sigma:[-1,1]\rightarrow \mathbf  M$ is a strongly $\epsilon$-bounded  $\mathcal C^r$ curve and we let $\hat \sigma:[-1,1]\rightarrow \mathbb P T \mathbf M$ be the associated induced map. \\
 
Then for some universal constant $C_r>0$ depending only on $r$ and  for any pair of integers $(k,k')$ there is a family $\Theta$ of affine maps from $[-1,1]$ to itself satisfying:
\begin{itemize}
\item $\hat\sigma^{-1}\left(\left\{\hat x\in \mathbb P T \mathbf M, \ k_g(x)=k \text{ and }k'_g(\hat x)=k'\right\}\right)\subset \bigcup_{\theta\in \Theta} \theta([-1,1])$,
\item $\forall \theta \in \Theta$, the curve $g\circ \sigma \circ \theta$ is bounded,
\item $\forall \theta \in \Theta, \ |\theta'|\leq e^{\frac{k'-k-1}{r-1}}/4$,
\item $\sharp \Theta \leq C_re^{\frac{k-k'}{r-1}}$.
\end{itemize} 
 
 \end{lem}

\begin{proof}

% and let $\mathbf k^{n+1}$ with $\mathcal H(\mathbf k^{n+1})\cap \left(\sigma \circ \theta_{\mathbf i^n}\right)_*\neq \emptyset$. 

\underline{\textit{First step :}} \textbf{Taylor polynomial approximation.} One   computes for an affine map $\theta:[-1,1]\circlearrowleft$ with contraction rate $b$ precised later and   with $y= \sigma(t)$,    $k_g(y)=k$, $k'_g(y)=k' $, $t\in \theta([-1,1])$:
\begin{align*}\|d^r(g\circ \sigma\circ \theta)\|_\infty &\leq b^r \left\|d^{r}\left(g_{2\epsilon}^y \circ \sigma_{2\epsilon}^y\right)\right\|_\infty , \textrm{with $\sigma_{2\epsilon}^y:=(2\epsilon)^{-1}\exp_y^{-1}\circ \sigma=2\epsilon^{-1}\left(\sigma(\cdot )-y\right)$,}\\
&\leq  b^r\left\|d^{r-1}\left( d_{\sigma_{2\epsilon}^y}g_{2\epsilon}^y\circ d\sigma_{2\epsilon}^y \right) \right\|_\infty,\\
&\leq b^r 2^r  \max_{s=0,\cdots,r-1}\left\|d^s\left(d_{\sigma_{2\epsilon}^y}g_{2\epsilon}^y\right)\right\|_{\infty}\max_{k=1,\cdots ,r}\|d^k\sigma_{2\epsilon}^y \|_\infty.
\end{align*}
By assumption on $\epsilon$, we have $\|d^s g_{2\epsilon}^y\|_{\infty}\leq 3\epsilon\|d_y g\|$ for any $r\geq s\geq 1$.
Moreover  $\max_{k=1,\cdots ,r}\|d^k\sigma_{2\epsilon}^y \|_\infty\leq 1$ as $\sigma$ is strongly $\epsilon$-bounded. 
Therefore by  Fa\'a di Bruno's formula, we get  for some\footnote[4]{Although these constants may differ at each step, they are all denoted by $C_r$.}  constants $C_r>0$ depending only on $r$:  
\begin{align*}\max_{s=0,\cdots,r-1}\|d^s\left(d_{\sigma_{2\epsilon}^y}g_{2\epsilon}^y\right)\|_{\infty} &\leq \epsilon C_r\|d_y g\|,\\
\text{then }&,\\
\|d^r(g\circ \sigma\circ \theta)\|_\infty &\leq \epsilon C_rb^r \|d_y g\|\max_{k=1,\cdots ,r}\|d^k\sigma_{2\epsilon}^y \|_\infty,\\
&\leq  C_rb^r \|d_y g\|\|d\sigma \|_\infty, \\
&\leq ( C_r b^{r-1}\|d_yg\|) \|d(\sigma \circ \theta)\|_{\infty}, \\
&\leq (C_r b^{r-1}e^{k}) \|d(\sigma \circ \theta)\|_{\infty},  \textrm{ because $k(y)=k$ }, \\
& \leq e^{k'-4}\|d(\sigma\circ \theta)\|_\infty, \textrm{ by taking  $b=\left(C_re^{k-k'+4 }\right)^{-\frac{1}{r-1}}$.}
\end{align*}

Therefore the Taylor polynomial  $P$ at $0$ of degree $r-1$  of $d(g\circ \sigma\circ \theta)$  satisfies on $[-1,1]$:
\begin{align*}
\|P-d(g\circ \sigma\circ \theta)\|_{\infty}&\leq e^{k'-4}\|d(\sigma\circ \theta)\|_\infty.
\end{align*}
We may cover $[-1,1]$ by at most $b^{-1}+1$ such affine maps $\theta$. \\

\underline{\textit{Second step :}} \textbf{Bezout theorem.}
Let $a=e^{k'}\|d(\sigma\circ \theta)\|_\infty$. Note that for $s\in [-1,1]$ with $k(\sigma \circ \theta(s))=k $ and $k'(\sigma \circ \theta(s))=k'$
we have $\|d(g\circ \sigma\circ \theta)(s)\|\in [ae^{-2},ae^{2}]$, therefore $\|P(s)\|\in [ae^{-3},ae^3]$. Moreover if we have now  $\|P(s)\|\in [ae^{-3},ae^3]$ for some $s\in [-1,1]$ we  get also $\|d(g\circ \sigma\circ \theta)(s)\|\in [ae^{-4},ae^{4}]$.

 By Bezout theorem the semi-algebraic set $\{ s\in [-1,1],\  \|P(s)\|\in  [e^{-3}a, e^{3}a]\}$ is the disjoint  union of closed   intervals $(J_i)_{i\in I}$ 
with $\sharp I$ depending only on $r$. Let $\theta_i$ be the composition of $\theta$ with an affine  reparametrization from $[-1,1]$ onto $J_i$. \\

\underline{\textit{Third step :}} \textbf{ Landau-Kolmogorov inequality.}
By the Landau-Kolmogorov inequality on the interval  (see Lemma 6 in  \cite{bur}), we have for some   constants $C_r\in \mathbb N^*$  and for all $1\leq s\leq r$:
\begin{align*}
\|d^s(g\circ \sigma\circ \theta_i)\|_\infty & \leq  C_r\left(\|d^r(g\circ \sigma\circ \theta_i)\|_\infty +\|d(g\circ \sigma\circ \theta_i)\|_\infty\right),\\
&\leq C_r\frac{|J_i|}{2}\left( \|d^r(g\circ \sigma\circ \theta)\|_\infty+ \sup_{t\in J_i}\|d(g\circ \sigma\circ \theta)(t)\| \right),\\
&\leq C_r a\frac{|J_i|}{2}.
\end{align*}
We cut again each $J_i$ into $1000C_r$ intervals $\tilde{J_i}$ of the same length with $$ \theta(\tilde{J}_i)\cap \sigma^{-1}\left\{x, \ k_g(x)=k \text{ and }k'_g(x)=k'\right\}\neq \emptyset.$$ Let $\tilde{\theta_i}$ be the affine reparametrization   from $[-1,1]$ onto  $\theta(\tilde{J_i})$. We check that $g\circ \sigma\circ \tilde{\theta_i}$ is bounded:
\begin{align*}
\forall s=2,\cdots, r, \   \|d^s(g\circ \sigma\circ \tilde{\theta_i})\|_\infty & \leq (1000C_r)^{-2} \|d^s(g\circ \sigma\circ \theta_i)\|_\infty,\\
&\leq \frac{1}{6}(1000C_r)^{-1}\frac{|J_i|}{2}a_ne^{-4},\\
&\leq  \frac{1}{6}(1000C_r)^{-1}\frac{|J_i|}{2}\min_{s\in J_i}\|d(g\circ \sigma\circ \theta)(s)\|,\\
&\leq  \frac{1}{6}(1000C_r)^{-1}\frac{|J_i|}{2}\min_{s\in \tilde{J}_i}\|d(g\circ \sigma\circ \theta)(s)\|,\\
&\leq \frac{1}{6} \|d(g\circ \sigma\circ \tilde{\theta_i})\|_\infty.
\end{align*}
This conclude the proof with $\Theta$ being the family of all $\tilde{\theta_i}$'s.
\end{proof}

We recall now a useful property of bounded curve (see Lemma 7 in \cite{burens} for a proof).

\begin{lem}\label{inter}
Let $\sigma:[-1,1]\rightarrow \mathbf M$ be a $\mathcal C^r$ bounded curve and let $B$ be  a ball of radius less than $\epsilon$. Then there exists an affine map $\theta:[-1,1]\circlearrowleft$ such that :
\begin{itemize}
\item  $\sigma\circ \theta$ is strongly $3\epsilon$-bounded,
\item  $\theta([-1,1])\supset \sigma^{-1}B$.
\end{itemize} 
\end{lem}

\subsection{Choice of the parameters $\kappa$ and $\epsilon_q$}
For a diffeomorphism $f:\mathbf M \circlearrowleft$ the scale $\epsilon_f$ in Lemma \ref{nondyn} may be chosen such that $\epsilon_{f^k}\leq \epsilon_{f^l}\leq \max(1, \|df\|_\infty)^{-k}$ for any $q\geq k\geq l\geq 1$. 
We take $\kappa =\frac{\epsilon_f}{36}$ and we choose $ \epsilon_q<\frac{\epsilon_{f^q}}{3}$ such that for any $\hat x, \hat y\in \mathbb PT\mathbf M$ which are $\epsilon_q$-close and for any $0\leq l\leq q$:
\begin{align}\label{eq}
\left| k_{f^l}( x ) - k_{f^l}(y )\right|&\leq 1,\\
  \left| k'_{f^l}( \hat x ) - k'_{f^l}(\hat y)\right|&\leq 1.\nonumber
\end{align}   Without loss of generality we can assume the local unstable curve $D$ (defined in Subsection \ref{zeta}) is reparametrized by a $\mathcal C^r$ strongly $\epsilon_q$-bounded map $\sigma:[-1,1]\rightarrow D$. 

Let $F_n$ be an atom of the partition  $ Q^{\mathtt E_n^M}$ and let $E_n^M=E_n^M(x)$ for any $\hat   x\in F_n$. Recall that the diameter of $Q$ is less than $\epsilon_q$. It follows from (\ref{eq}) that for any $\hat x\in F_n$ we have with $\hat{\zeta}_{F_n}^{M}=\int \mu_{\hat x,n}^M\, d\zeta_{F_n}(x)$:
\begin{align*}
\sum_{l\in E_n^M } \left| k_{f^q}(f^l x)-k'_{f^q}(F^l\hat x)\right| \leq 10\sharp E_n^M+ \int \log ^+\|d_yf^q\|\, d\zeta_{F_n}^{M}(y)-\int \phi_q\, d\hat{\zeta}_{F_n}^{M}.
\end{align*}

Therefore we may fix some   $0\leq c<q$, such that for any $x\in F_n$ 
\begin{align*}\sum_{l\in (c+q\mathbb N)\cap E_n^M } \left| k_{f^q}(f^l x)-k'_{f^q}(F^l\hat x)\right|& \leq 10\frac{n}{q}+ \frac{1}{q}\left(\int \log ^+\|d_yf^q\|\, d\zeta_{F_n}^{M}(y)-\int \phi_q\, d\hat{\zeta}_{F_n}^{M}\right), \\
&\leq   10\frac{n}{q}+2A_f\frac{qn}{M}+\frac{1}{q}\int \log ^+\|d_yf^q\|\, d\zeta_{F_n}^{M}(y)-\int \phi\, d\hat{\zeta}_{F_n}^{M}.
\end{align*}

\subsection{Combinatorial aspects}
 We put  $\partial_lE_n^M:=\{ a\in E_n^M \text{ with }a-1\notin E_n^M\}.$ Then we let $\mathcal A_n:=\{0=a_1<a_2<\cdots a_m\}$
 be the union of  $\partial_l E_n^M$,  $[0,n[\setminus E_n^M$ and $(c+q\mathbb N)\cap [0,n[ $.
We also let $b_i=a_{i+1}-a_i$ for $i=1,\cdots , m-1$ and $b_m=n-a_m$.

For a sequence $\mathbf k= (k_l,k'_l)_{ l\in \mathcal A_n}$ of integers, a positive integer $m_n$ and  a subset $\overline{E}$ of $[0,n[$, we let $F_n^{\mathbf k , \overline{E},m_n}$ be  the subset of  points  $\hat x\in F_n$ satisfying:
\begin{itemize}
\item  $\overline{E}=E_n(x)\setminus E_n^M(x)$, 
\item $k_{a_i}=k_{f^{b_i}}(f^{a_i} x)$ and  $k'_{a_i}=k'_{f^{b_i}}(F^{a_i}\hat x)$  for $i=1,\cdots, m$, 
\item $m_n(x)=m_n$.
\end{itemize}

\begin{lem}\label{comb}
$$\sharp\left\{(\mathbf k, \overline E,m_n), \ F_n^{\mathbf k , \overline{E},m_n}\neq \emptyset\right\}\leq ne^{2nA_fH(A_f^{-1})} 3^{n(1/q+1/M)}e^{nH(1/M)}.$$

\end{lem}
\begin{proof}
Firstly observe that if $a_{i}\notin E_n^M$ then $b_i=1$. In particular 
$\sum_{i, \ a_i\notin E_n^M}k_{a_i}\leq (n-\sharp E_n^M)\log^+\|df\|_\infty \leq (n-\sharp E_n^M)(A_f-1) $. The number of such sequences $(k_{a_i})_{i, \ a_i\notin E_n^M}$ is therefore bounded above by $\binom{r_nA_f}{r_n}$  with $r_n=n-\sharp E_n^M$and its logarithm is dominated  by $r_nA_fH(A_f^{-1})+1\leq nA_fH(A_f^{-1})+1 $. Similarly the number of sequence $(k'_{a_i})_{i, \ a_i\notin E_n^M}$ is less than $nA_fH(A_f^{-1 })+1$. 

Then from the choice of $\epsilon_q$  in  (\ref{eq}) there are at most three possible values  of $k_{a_i} (x)$ for  $a_i\in E_n^M$ and $x\in F_n$.

Finally as $\sharp \overline{E}\leq n/M$, the number of admissible sets $\overline{E}$ is less than $\binom{n}{[n/M]}$ and thus its logarithm  is bounded above by $nH(1/M)+1$. Clearly we can also fix the value of $m_n$ up to a factor $n$.

\end{proof}

\subsection{The induction}

We fix $\mathbf k$, $m_n$ and $\overline{E}$ and we reparametrize appropriately the set  $F_n^{\mathbf k, \overline{E},m_n}$. 

\begin{lem}\label{induc}
With the above notations there are  families $(\Theta_i)_{i\leq m}$ of affine maps from $[-1,1]$ into itself such that :
\begin{itemize}
\item $\forall \theta\in \Theta_i \ \forall j\leq i$ the curve $f^{a_i}\circ \sigma\circ \theta$ is strongly $\epsilon_{f^{b_i}}$-bounded,  
\item $\hat{\sigma}^{-1}\left(F_n^{\mathbf k, \overline E,m_n}\right)\subset \bigcup_{\theta\in \Theta_i} \theta([-1,1])$, 
\item  $\forall \theta_i\in \Theta_i  \ \forall j< i, \exists \theta^i_j\in \Theta_j, \  \frac{|\theta'_i|}{|(\theta_j^i)'|}\leq \prod_{j\leq l< i}e^{\frac{k'_{a_l}-k_{a_l}-1}{r-1}}/4$,
\item $\sharp \Theta_i\leq C\max\left(1,\|df\|_\infty\right)^{\sharp \overline E\cap [1,a_i] } \prod_{j< i}C_re^{\frac{k_{a_j}-k'_{a_j}}{r-1}}$.
\end{itemize}
\end{lem}

\begin{proof}
We argue by induction on $i\leq m$. By changing the constant $C$,  it is enough to consider $i$ with $a_i>N$. Recall that the integer $N$ was chosen in such a way that for any $x\in \mathtt F $ there is $0\leq k\leq N$ with $F^k\hat x\in H_\delta$. We assume the family $\Theta_i$ for $i<m$ already built and we will define $\Theta_{i+1}$. Let $\theta_i \in \Theta_i$. We apply Lemma \ref{nondyn} to the strongly $\epsilon_{f^{b_i}}$-bounded curve $f^{a_i}\circ \sigma\circ \theta_i$ with $g=f^{b_i}$.   Let $\Theta$ be the family of affine reparametrizations of $[-1,1]$ satisfying the conclusions of Lemma \ref{nondyn}, in particular $f^{a_{i+1}}\circ \sigma\circ \theta_i\circ \theta$ is bounded, $|\theta'|\leq e^{\frac{k'_{a_i}-k_{a_i}-1}{r-1}}/4 $ for all $\theta \in \Theta$ and $\sharp \Theta\leq C_re^{\frac{k_{a_i}-k'_{a_i}}{r-1}}$. We distinguish three cases:
\begin{itemize}
\item \underline{$a_{i+1}\in E_n^M$.} The diameter of $F^{a_{i+1}}F_n$ is less than $\epsilon_q\leq \frac{\epsilon_{f^{b_{i+1}}}}{3}$. By Lemma \ref{inter} there is an affine map $\psi:[-1,1]\circlearrowleft$ such that $f^{a_{i+1}}\circ \sigma\circ \theta_i\circ \theta\circ \psi$ is strongly $\epsilon_{f^{b_{i+1}}}$-bounded and its image contains    the intersection of the bounded curve $f^{a_{i+1}}\circ \sigma\circ \theta_i\circ \theta$ with $f^{a_{i+1}}F_n$. We let 
 then $\theta_{i+1}=\theta_i\circ \theta \circ \psi\in \Theta_{i+1}$.
\item \underline{$a_{i+1}\in E\setminus E_n^M$}. Observe that $b_{i+1}=1$, therefore $\epsilon_{f^{b_i}}\leq \epsilon_{f^{b_{i+1}}}$. Then  the length of the curve $f^{a_{i+1}}\circ \sigma\circ \theta_i \circ \theta$  is less than $3\|df\|_\infty \epsilon_{f^{b_i}}$, thus may be covered by $[3\|df\|_\infty]+1$ balls of radius less than $\epsilon_{f^{b_{i+1}}}$. We then  use Lemma \ref{inter} as in the previous case  to reparametrize the intersection of this curve with each ball by 
a strongly $\epsilon_{f^{b_{i+1}}}$-bounded curve. We define in this way the associated  parametrizations of $\Theta_{i+1}$. 
\item \underline{$a_{i+1}\notin E$ and $a_{i+1}\notin E_n^M$}. We claim that  $\|d(f^{a_{i+1}}\circ \sigma\circ \theta_i\circ \theta\|\leq \epsilon_f/6$. Take $\hat x \in  F_n^{\mathbf k, \overline E, m_n}$ with $x=\pi(\hat x)=\sigma \circ \theta_i\circ \theta(s)$. Let   $K_x=\max\{ k < a_{i+1},  \ F^k\hat x\in H_\delta\}\geq N$.  
 Observe that $[K_x, a_{i+1}]\cap E_n^M=\emptyset$, therefore for $K_x\leq a_l< a_{i+1}$, we have $b_l=1$, then  $a_l=a_{i+1}-i-1+l$. We argue by contradiction by assuming : \begin{align}\label{mord}\|d(f^{a_{i+1}}\circ \sigma\circ \theta_i\circ \theta\|\geq \epsilon_f/6=6\kappa
 \end{align} By Lemma \ref{nonam},  the point $f^{a_{i+1}} x$ belongs to $G_\kappa$. We will show $F^{a_{i+1}}\hat x\in H_\delta$. Therefore we will get $F^{a_{i+1}}\hat x\in  G=\pi^{-1}G_\kappa\cap H_\delta$ contradicting $a_{i+1}\notin E$. To prove $F^{a_{i+1}}\hat x\in H_\delta$ it is enough to show $\sum_{j\leq l<a_{i+1}} \psi(F^l\hat x)\geq (a_{i+1}-j)\delta$ for any $K_x\leq j<a_{i+1}$ because $F^{K_x}(\hat x)$ belongs to $H_\delta$. 
 For any $K_x\leq j< a_{i+1}$ we have :
\begin{align}\label{frousse}
\|d(f^{a_{i+1}}\circ \sigma\circ \theta_i\circ \theta\|_\infty&\leq 2
 \|d_s(f^{a_{i+1}}\circ \sigma\circ \theta_i\circ \theta\|,\text{ because  $f^{a_{i+1}}\circ \sigma\circ \theta_i\circ \theta$ is bounded,}\nonumber\\
&\leq 2\| d_{f^jx}f^{a_{i+1}-j}(\hat x) \|\times \|d_s(f^{a_{\overline j}}\circ \sigma \circ \theta_{\overline j}^i)\|\times \frac{|\theta'_i|\times|\theta'|}{|(\theta_{\overline  j}^i)'|}, \text{ with $a_{\overline j}=j$,} \nonumber \\
& \leq \frac{\epsilon_f}{3} \| d_{f^jx}f^{a_{i+1}-j}(\hat x) \| \prod_{\overline{j}\leq l\leq i}e^{\frac{k'_{a_l}-k_{a_l}-1}{r-1}}/4\text{ by induction hypothesis}, \nonumber \\
\frac{1}{2}&\leq \| d_{f^jx}f^{a_{i+1}-j}(\hat x) \| \prod_{\overline{j}\leq l\leq i}e^{\frac{k'_{a_l}-k_{a_l}-1}{r-1}}/4 \text{ by assumption (\ref{mord})}.
\end{align}

Recall again that  for $\overline{j}\leq l \leq i$, we have $b_l=1$, thus 
$$\left| k_{a_l}-\log \|d_{f^{a_l}x }f\| \right|\leq 1$$ and 
$$k'_{a_l}\leq \phi(F^{a_l}\hat x).  $$ 
Therefore we get for any $K_x\leq j< a_{i+1}$ from (\ref{frousse}): 

\begin{align*}2^{a_{i+1}-j} &\leq  e^{\frac{r}{r-1} \sum_{j\leq l<a_{i+1}} \phi(F^l\hat x)} e^{-\frac{1}{r-1} \sum_{j\leq l<a_{i+1}} \log^+ \|d_{f^{l}x }f\|},\\
(a_{i+1}-j)\log 2&\leq  \frac{r}{r-1} \sum_{j\leq l<a_{i+1}} \psi(F^l\hat x), \text{ by definition of $\psi$,}\\
(a_{i+1}-j)\delta&\leq  \sum_{j\leq l<a_{i+1}} \psi(F^l\hat x) \text{, as $\delta$ was chosen less than $\frac{r-1}{r}\log 2$}.
\end{align*}

%\label{estim}& \leq \frac{\epsilon}{3} \| d_{f^jx}f^{a_{i+1}-j}(\hat x) \| \prod_{j\leq l< a_{i+1}}e^{\frac{k'_{l}-k_{l}}{r-1}} \nonumber, \\1&\leq 2 \| d_{f^jx}f^{a_{i+1}-j}(\hat x) \| \prod_{j\leq l< a_{i+1}}e^{\frac{k'_{l}-k_{l}}{r-1}}.\end{align}

\end{itemize}

\end{proof}
\begin{lem}\label{lastly}
$$\sum_{i, \ m_n>a_i\notin E_n^M}\frac{k_{a_i} -k'_{a_i}}{r-1}\leq \left(n-\sharp E_n^M\right)\left(\frac{\log^+ \|df\|_\infty}{r}+\frac{1}{r-1}\right).$$ 
\end{lem}
\begin{proof}
The intersection of $[0,m_n[$ with the complement set of $E_n^M$ is the disjoint union of neutral blocks and  possibly an interval of integers of the form $[l,m_n[$. In any case $F^{\mathtt j}\hat x$ belongs to $H_\delta$ for  such an interval $[\mathtt i, \mathtt j [$ for any  $x\in F_n^{\mathbf k , \overline{E},m_n}$. In particular, we have 
\begin{align*}\sum_{l, a_l\in [\mathtt i, \mathtt j [ }k'_{a_i}-\frac{k_{a_i}}{r} & \geq (\delta-1)(\mathtt j -\mathtt i )
\end{align*} therefore
\begin{align*}
\sum_{i, \ m_n>a_i\notin E_n^M}k'_{a_i}-\frac{k_{a_i}}{r}& \geq -(n-\sharp E_n^M),\\
\sum_{i, \ m_n>a_i\notin E_n^M} \frac{k_{a_i}-k'_{a_i}}{r-1} & \leq \frac{n-\sharp E_n^M}{r-1}+ \frac{\sum_{i, \ m_n>a_i\notin E_n^M} k_{a_i}}{r},\\
&\leq  \left(n-\sharp E_n^M\right)\left(\frac{\log^+ \|df\|_\infty}{r}+\frac{1}{r-1}\right).
\end{align*}

\end{proof}

\subsection{Conclusion}

We let $\Psi_n$ be the family of $\mathcal C^r$ curves $\sigma\circ \theta$ for   $\theta\in \Theta_m=\Theta_m(\mathbf k, \overline E, m_n)$ with $\Theta_m$ as in  Lemma \ref{induc} over all admissible parameters $\mathbf k, \overline E, m_n$. For $\theta \in \Theta_m$ the curve $f^{a_i}\circ \sigma\circ \theta$  is strongly $\epsilon_{f^{b_i}}$-bounded for any $i=1,\cdots, m$, in particular 

$$\forall i=1,\cdots, m, \ \|d(  f^{a_i}\circ \sigma\circ \theta)\|_\infty\leq \epsilon_{f^{b_i}}\leq \max(1,\|df\|_\infty)^{-b_i},$$
therefore 
$$\forall j=0,\cdots, n, \ \|d(  f^{j}\circ \sigma\circ \theta)\|_\infty\leq 1.$$ 

 By combining the previous estimates, we get moreover:

\begin{align*}
\sharp \Psi_n& \leq  \sharp\left\{(\mathbf k, \overline E,m_n), \ F_n^{\mathbf k , \overline{E},m_n}\neq \emptyset\right\} \times \sup_{\mathbf k, \overline{E},m_n}\sharp \Theta_n(\mathbf k, \overline E, m_n), \\
&\leq ne^{2(n-\sharp E_n^M)A_fH(A_f)} 3^{n(1/q+1/M)}e^{nH(1/M)}\sup_{\mathbf k, \overline{E},m_n}\sharp \Theta_n(\mathbf k, \overline E, m_n), \text{ by Lemma \ref{comb}, }\\
&\leq  ne^{2(n-\sharp E_n^M)A_fH(A_f)} 3^{n(1/q+1/M)}e^{nH(1/M)}  \max(1,\|df\|_\infty)^{\sharp \overline E } \prod_{j\leq m}C_re^{\frac{k_{a_j}-k'_{a_j}}{r-1}}, \text{ by Lemma \ref{induc}.}\\
\end{align*}
 Then we decompose the product into four terms :
 \begin{itemize}
 \item $\sum_{i, \ m_n>a_i\notin E_n^M}\frac{k_{a_i} -k'_{a_i}} {r-1}\leq (n-\sharp E_n^M)\left( \frac{\log^+ \|df\|_\infty}{r}+\frac{1}{r-1}\right) $ by Lemma \ref{lastly}, 
 \item $\sum_{i, \ m_n\leq a_i}\frac{k_{a_i} -k'_{a_i}}{r-1}\leq  (n-m_n)\frac{A_f}{r-1}$, 
 \item $\sum_{i, a_i\in E_n^M\cap (c+q\mathbb N ) }\frac{k_{a_i} -k'_{a_i}}{r-1}\leq 10\frac{n}{q}+2A_f\frac{qn}{M}+ \frac{1}{r-1}\left(\int \frac{\log ^+\|d_yf^q\|}{q}\, d\zeta_{F_n}^{M}(y)-\int \phi\, d\hat{\zeta}_{F_n}^{M}\right),$\\
 \item $\sum_{i, a_i\in E_n^M\setminus (c+q\mathbb N ) }\frac{k_{a_i} -k'_{a_i}}{r-1} \leq 2A_f\frac{qn}{M}$.\\
 \end{itemize}
By letting 
$$B_r=\frac{1}{r-1}+\log C_r,$$

$$\gamma_{q,M}(f):=2\left(\frac{1}{q}+\frac{1}{M}\right)\log C_r+H(1/M)+\frac{10+\log 3}{q}+\frac{4qA_f+\log 3}{M},$$

$$\tau_n=\sup_{x\in \mathtt F}\left(1-\frac{m_n(x)}{n}\right)\frac{A_f}{r-1}+\frac{\log (nC)}{n},$$

we get with $C(f):=2A_fH(A_f^{-1})+\frac{\log^+ \|df\|_\infty}{r}+B_r$:
 \begin{align*}\frac{1}{n}\log \sharp \Psi_{F_n}\leq &\left(1-\frac{\sharp E_n^M}{n}\right)C(f)\\ &+\left(\log 2+\frac{1}{r-1}\right)\left(\int \frac{\log ^+\|d_xf^q\|}{q}\, d\zeta_{F_n}^{M}(x)-\int \phi\, d\hat{\zeta}_{F_n}^{M}\right)\\&+\gamma_{q,M}(f)+\tau_n,
 \end{align*}
 This concludes the proof of Proposition \ref{paraa}.

\appendix
\section*{Appendix}
We explain in this appendix how our Main Theorem implies Buzzi-Crovisier-Sarig statement.\\

 Let $(f_k)_k$, $(\nu_k^+)_k$ and   $\hat \mu$ be as in the setting of Theorem \ref{cochon}. 
Then, either $\lim_k\lambda^+(\nu_k)=\int \phi\, d\hat \mu\leq \frac{\lambda^+(f)}{r}$ and we get by Ruelle inequality, $\limsup_kh(\nu_k)\leq \frac{\lambda^+(f)}{r}$ or there exists  $\alpha\in \left] \frac{\lambda^+(f)}{r}, \min\left(\int \phi\, d\hat \mu,\frac{\lambda^+(f)}{r-1} \right)\right[$.  By applying our Main Theorem with respect to $\alpha$, there is  a  decomposition $\hat \mu=(1-\beta_\alpha)\hat \mu_{0,\alpha}+\beta_\alpha \hat \mu_{1,\alpha}^+$ satisfying 
$\limsup_{k\rightarrow +\infty} h(\nu_k)\leq \beta_\alpha h(\mu_{1,\alpha})+(1-\beta_\alpha)\alpha$. But it follows from the proofs that  $\beta_\alpha \mu_{1,\alpha}$ is a component of $\beta \mu_1$ with $\beta $ and $\mu_1$ being as in Buzzi-Crovisier-Sarig's statement as they consider empirical measure associated to a larger set $G$  (see Remark \ref{mieux}).  In particular 
$\beta_\alpha h(\mu_{1,\alpha})\leq \beta h(\mu_1)$, therefore $\limsup_{k\rightarrow +\infty} h(\nu_k)\leq\beta h(\mu_1)+\frac{\lambda^+(f)+\lambda^+(f^{-1})}{r-1}$. \\

In Theorem C \cite{BCS2}, the authors also proved $\int \phi \,d\hat \mu_0=0$ whenever $\beta \neq 1$. Therefore we get here $(1-\beta_\alpha)\int \phi\, d\hat \mu_{0,\alpha}\geq (1-\beta)\int \phi \,d\hat \mu_0=0$, then $\int \phi\, d\hat\mu_{0,\alpha}\geq 0$. But maybe we could have $\int \phi\, d\hat\mu_{0,\alpha}> 0$.

\end{document}